%% file: harmonic-siaga-final-submission.tex
\documentclass[
onefignum,onetabnum]{siamonline190516}

\input{ex_shared}

\ifpdf
\hypersetup{
  pdftitle={Harmonic Persistent Homology},
  pdfauthor={Saugata Basu and  Nathanael Cox}
}
\fi

\newcommand{\im}{\mathrm{Im}}

\newcommand{\ds}{\displaystyle}

\usepackage{amssymb,color,epsf}
\usepackage{mathtools}

\usepackage{bm}
\usepackage{amsmath}
\usepackage{tabularx}
\usepackage{color, colortbl}
\usepackage{url}
\usepackage{enumitem}

\DeclareMathAlphabet{\mathpzc}{OT1}{pzc}{m}{it}
\usepackage[margin=1.0in]{geometry}

\usepackage [cmtip,arrow]{xy}
\xyoption{all}
\usepackage {pb-diagram,pb-xy}
\usepackage[mathscr]{eucal}
\usepackage[utf8]{inputenc}
\usepackage[T1]{fontenc}
\usepackage[english]{babel}
\usepackage{bm}
\usepackage{tabularx}

\usepackage{bm}

\usepackage{blkarray, bigstrut}
\usepackage{gauss}

\newcommand{\sqbracket}[1]{[#1]}

\def\theenumi{\Alph{enumi}}
\def\theenumii{\arabic{enumii}}
\def\theenumiii{\roman{enumiii}}

\newcommand {\junk}[1]{}

\newcommand {\R} {\mathbb{R}}

\newcommand {\Z}  {\mathbb{Z}}

\newcommand {\la}   {{\langle}}
\newcommand {\ra}   {{\rangle}}

\newcommand {\dist} {{\rm dist}}

\newcommand {\Ker}      {\mathrm{Ker}}

\newcommand {\PP}     {\mathbb{P}} 

\newcommand {\spanof} {\mathrm{span}}
\newcommand{\card}{\mathrm{card}}
\newcommand{\rank}{\mathrm{rank}}

\def\addots{\mathinner{\mkern1mu
\raise1pt\vbox{\kern7pt\hbox{.}}
\mkern2mu\raise4pt\hbox{.}\mkern2mu
\raise7pt\hbox{.}\mkern1mu}}

\newcommand{\HH}  {\mbox{\rm H}}

\newcommand{\image}{\mathrm{Im}}

\newcommand{\hide}[1]{}

\newcommand{\supp}{\mathrm{supp}}

\newcommand{\nc}{\newcommand}
\newcommand{\rc}{\renewcommand}
\nc{\mc}{\mathcal}
\rc{\t}{\text}
\nc{\op}[1]{\operatorname{#1}}
\nc{\opcat}[1]{\mathbf{#1}}

\nc{\id}{\op{id}}

\nc{\umutnote}[1]{{\marginpar{\small \textcolor{blue}{#1}}}}

\nc{\cA}{\mc{A}}\nc{\cB}{\mc{B}}\nc{\cC}{\mc{C}}\nc{\cD}{\mc{D}}\nc{\cE}{\mc{E}}\nc{\cF}{\mc{F}}\nc{\cG}{\mc{G}}\nc{\cH}{\mc{H}}\nc{\cI}{\mc{I}}\nc{\cJ}{\mc{J}}\nc{\cK}{\mc{K}}\nc{\cL}{\mc{L}}\nc{\cM}{\mc{M}}\nc{\cN}{\mc{N}}\nc{\cO}{\mc{O}}\nc{\cP}{\mc{P}}\nc{\cQ}{\mc{Q}}\nc{\cR}{\mc{R}}\nc{\cS}{\mc{S}}\nc{\cT}{\mc{T}}\nc{\cU}{\mc{U}}\nc{\cV}{\mc{V}}\nc{\cW}{\mc{W}}\nc{\cX}{\mc{X}}\nc{\cY}{\mc{Y}}\nc{\cZ}{\mc{Z}}

\rc{\PP}{\mathbb{P}}
\rc{\AA}{\mathbb{A}}
\nc{\bbC}{\mathbb{C}}
\nc{\CC}{\mathbb{C}}

\nc{\code}[1]{{\texttt{#1}}}
\nc{\mcode}[1]{{\text{\texttt{#1}}}}
\nc{\xto}[1]{\raisebox{-0.03cm}{\scalebox{0.85}{$\,\xrightarrow{#1}\,$}}}
\nc{\xtonormal}[1]{\xrightarrow{#1}}
\nc{\xfrom}[1]{\xleftarrow{#1}}
\nc{\sidenote}[1]{\marginpar{\small #1}}

\nc{\Aff}{\opcat{Aff}}
\nc{\AffVar}{\opcat{AffVar}}
\nc{\ProjVar}{\opcat{ProjVar}}
\nc{\GAP}{\opcat{GrAlgPairs}}
\nc{\GA}{\opcat{GrAlg}}

\nc{\acc}{\mathrm{a.c.c}}
\nc{\GL}{\mathrm{GL}}

\nc{\Mod}{\t{-}\opcat{Mod}}
\nc{\Sub}{\opcat{Sub}}
\nc{\iso}{\cong}
\nc{\compose}{\circ}

\newcommand{\FF}{\mathcal{F}}
\newcommand{\h}{\mathcal{H}}
\newcommand{\spn}{\mathrm{span}}
\newcommand{\proj}{\mathrm{proj}}
\newcommand{\bp}{\begin{proof}}
\newcommand{\ep}{\end{proof}}

\newcommand{\Gr}{\mathrm{Gr}}
\newcommand{\persistent}{\mathrm{persistent}}
\newcommand{\Rep}{\mathrm{Rep}}
\newcommand{\econt}{\mathrm{content}}

\newcommand\restr[2]{{
  \left.\kern-\nulldelimiterspace 
  #1 
  \vphantom{\big|} 
  \right|_{#2} 
  }}

\begin{document}

\setlength{\abovedisplayskip}{1.5pt}
\setlength{\belowdisplayskip}{1.5pt}

\maketitle




\begin{abstract}
We introduce harmonic persistent homology spaces for filtrations of finite simplicial complexes.
As a result we can associate concrete subspaces of cycles to each bar of the barcode of the filtration.
We prove stability of the harmonic persistent homology subspaces, as well as the subspaces associated
to the bars of the barcodes,
under small perturbations of functions
defining them. We relate
the notion of ``essential simplices'' introduced in an earlier work to identify simplices which play 
a significant role in the birth of a bar, with that of harmonic persistent homology.
We prove that the harmonic representatives of simple bars maximizes the ``relative essential content''
amongst all representatives of the bar, where the relative essential content is the weight a particular 
cycle puts on the set of essential simplices. 
\footnote{An extended abstract of the paper appeared in the Proceedings of the IEEE Symposium on the Foundations of Computer Science, 2021.}
\end{abstract}

\section{Introduction}
\label{sec:intro}
The main topic of this paper concerns the theory of \emph{persistent homology} which is a central object of interest in the burgeoning field of \emph{topological data analysis}.

\subsection{Background on persistent homology}
We begin with some motivation behind the introduction of persistent homology.
One way that simplicial complexes arise in topological data analysis is via the  Čech (or its closely related cousin, the Vietoris-Rips) complex \cite[pp. 60-61]{Edelsbrunner-Harer2010} 
associated to a point set. Let $X$ be a (finite) subset of some metric space which for concreteness let us assume to be  $\R^d$ (with its Euclidean metric).  In practice, $X$ may consist of a finite set of points  (often  called ``point-cloud data'')
which approximates some subspace or sub-manifold $M$ of $\R^d$. The topology (in particular, the homology groups) of the manifold $M$ is not reflected in the set
of points $X$ (which is a discrete topological space under the subspace topology induced from that of $\R^d$). 
Now for $r \geq 0$, let $X_r$ denote the union of closed Euclidean balls,
$B(x,r)$, of radius $r$ centered at the points $x \in X$. In particular, $X_0 = X$.
Also, for $0 \leq r \leq r'$, we have that $X_r \subset X_{r'}$. Thus,
$(X_r)_{r \in \R_{\geq 0}}$ is an increasing family of topological spaces
indexed by $r \geq 0$. This is an example of a (continuous) filtration of topological spaces. 
For each $r \geq 0$, we can associate a finite simplicial complex $K_r$ --
the  \emph{nerve complex} of the tuple of balls $(B(x,r))_{x \in X}$.
Informally, the simplicial complex $K_r$ has vertices indexed by the set
$X$, and for each subset $X'$ of $X$ of cardinality $p+1$, we include the 
$p$-dimensional simplex spanned by the vertex set corresponding to $X'$ if and only if 
\[
\bigcap_{x \in X'} B(x,r) \neq \emptyset.
\]
It is a basic result in 
algebraic topology (the ``nerve lemma'') that 
the simplicial homology groups, $\HH_*(K_r)$, are isomorphic to 
the  (say singular) homology groups, $\HH_*(X_r)$, of $X$. 
More precisely, the nerve lemma states that the geometric realization
$|K_r|$ is homotopy equivalent to (in fact, is a deformation retract of)  $X_r$ (homotopy equivalent spaces have isomorphic homology groups).
Observe that for $r \leq r'$, 
$K_r$ is a sub-simplicial complex of $K_{r'}$, and since there are only 
finitely many simplicial complexes on $\card(X)$-many vertices, there
are finitely many distinct simplicial complexes in the tuple $(K_r)_{r \geq 0}$. Thus, we obtain 
a finite  nested sequence $\mathcal{F}$ of simplicial complexes, 
$K_0 \subset K_{r_1} \subset K_{r_2} \subset \cdots \subset K_{r_n}$
in which each complex is a subcomplex of the next.
We will refer to $\mathcal{F}$ as a finite \emph{filtration} of 
simplicial complexes.

Let us return to the picture of the point-cloud $X$ approximating an underlying manifold $M$. The homology of the manifold is captured (by virtue of the nerve lemma) by the simplicial homology groups of the various simplicial complexes occurring in the finite filtration $\mathcal{F}$.
However, this correspondence is not bijective. As one can easily visualize, as $r$ starts growing from $0$, there are many spurious homology classes that are born and quickly die off (i.e. the corresponding holes are filled in) and these have nothing to do with the topology of $M$. Persistent homology is a tool that can be used to separate this ``noise'' from the  bona fide homology classes of $M$. The persistent homology of the filtration $\mathcal{F}$ is encoded as a set of intervals (called bars) in the barcode of the filtration $\mathcal{F}$ (see Definition~\ref{def:barcode2})).
Intervals (bars) of short length corresponds to noise, while the ones which
are long (persistent) reflect the homology of the underlying manifold $M$.
The barcode of the filtration associated to $X$ can be used as a feature of $X$ for learning or comparison purposes. In particular, the barcodes of two 
finite sets $X, X'$,  which are ``close'' as finite metric spaces, are themselves close under an appropriately defined notion of distance between barcodes. Such results (called stability theorems)  form the theoretical basis of practical applications of persistent homology, and we will state
new stability theorems later on in the paper.

\subsection{Associating cycles to bars}
As mentioned previously, the output of a persistent homology computation is often displayed as a ``barcode'' 
(see Example~\ref{eg:essential} and 
Figure~\ref{fig:essential} for an illustration). 
The barcode is considered an important invariant 
of the given filtration (or of the underlying metric space giving rise to the filtration).
We will define
precisely barcodes of filtrations later in the paper (see Definition~\ref{def:barcode2}).
For the moment it will suffice to note that the individual ``bars''  in the barcode of a filtration have some intuitive topological meaning (as explained above in the context of point-cloud data). They correspond
loosely speaking to the lifetime of homology classes appearing in the homology of the 
simplicial complexes that appear in the filtration (here we are thinking of the ordered index set of the filtration
as time). However, a new homology class that is ``born'' at a certain time  is defined only modulo a certain subspace in the homology of the complex at that time  -- thus identifying
a bar with a particular homology class is problematic.

Often in practice there is a demand to associate not just a homology class, but a specific 
\emph{cycle from the chain group representing this class} or at least a \emph{set of simplices} to each bar. 
This is because in applications
the simplices of the simplicial complexes of a filtration themselves  often have special significance. For instance, the vertices of a given simplicial complex could be labelled by
genes and a $p$-simplex $\sigma = (g_0,\ldots,g_p)$ may signify positive correlation between the genes 
$g_0,\ldots,g_p$ (say in causing a certain disease).
As an example, in \cite{lockwood2014}, the authors associate bars with representative cycles to determine how topological features correlate with genes that are associated with cancer biogenesis.

There have been several approaches to the problem of
associating specific cycle representatives to persistent homology classes.
Most of these approaches involve minimization of some weight on the space of cycles representing a homology class.
For instance,
volume-optimal cycles were proposed in the non-persistent setting in \cite{Dey-2010} and in the setting of persistent homology in \cite{volume}. Volume-optimal cycles are cycles of a homology class with the fewest number of simplices, and they can be found as solutions to a linear programming optimization problem \cite{Dey-2010}. 
In Dey et al.\ \cite{onecycles}
the authors give a polynomial time algorithm for computing an optimal representative for a given finite bar (interval) in the $p$-th persistence diagram of a filtration of a simplicial complex which is a weak $(p+1)$-pseudo-manifold.
Algorithms for computing such optimal cycles have been implemented -- see for instance \cite{Li-et-al}.
A different approach for selecting a representative cycle can be found in \cite{tracking}. The authors obtain a representative cycle by tracking when the addition or removal of a simplex causes a class to be born or die. 

In this paper we describe a new
approach based on the theory of \emph{harmonic chains}. We consider homology groups with coefficients in $\R$ and impose an inner product on the chain
group to make the chain groups an Euclidean space.
As a result we are able to identify the various persistent 
homology groups, as well as the bars in the barcode of the given filtration, as subspaces of the  
simplicial chain groups themselves. 
Note that in contrast with ordinary persistent homology theory, we are able to associate canonically (only depending on the chosen inner product) a certain subspace
of the chain space to each bar. When the bar is of multiplicity one
(see Definition~\ref{def:barcode2}), 
this subspace is spanned by a single vector, and we have a uniquely defined (up to scalar multiplication)
cycle representing the bar -- we call such a cycle a \emph{harmonic representative} of the bar. 
Intuitively, instead of selecting a representative cycle of the smallest
possible weight (as in \cite{Dey-2010,volume}), which might in fact omit
some simplices altogether, the harmonic representative
of a bar will tend to produce an ``average'' representative.

There are several reasons to consider harmonic representatives.
Instead of trying to optimize the length of the cycle the harmonic representatives 
put more relative weight on certain important simplices.
Since as remarked earlier, the simplices themselves in the simplicial complex underlying the filtration often have
domain dependent meaning --  if a particular simplex shows up with non-zero coefficient in \emph{every} cycle representing the homology class, then this fact may be considered
significant from an application point of view
(the lengths of representative cycles are not so significant
in these applications).
This idea was formalized in \cite{essential} where the notion
of \emph{essential} simplices corresponding to the bars of a barcode was introduced. 
Informally, a simplex is essential relative to a bar, if it occurs with a non-zero coefficient in \emph{every}  cycle
representing the bar. 
We generalize this notion and associate a set of essential simplices to any simple bar
(bars having multiplicity one -- see Definition~\ref{def:essential}). The harmonic representative of a 
bar will maximize amongst all representative cycles the relative weight of the essential simplices
(see Section~\ref{subsec:essential}). 

A second and perhaps more important justification of considering 
harmonic persistent homology is that the harmonic persistent homology  more accurately 
reflects the ``geometry'' of the filtrations on labelled simplicial complexes. 
We prove stability results  (see Section~\ref{subsec:stability}) -- 
harmonic persistent homology subspaces of simplicial filtrations which
are close will be close  (in a technical sense to be defined later) as elements of certain Grassmannians. 
In addition,
filtrations whose harmonic persistent homology are close as subspaces  will also have the harmonic representatives 
of their bars which are close (angle between corresponding subspaces will be small). 
Thus, bar diagrams augmented with the harmonic representatives of the bars is potentially 
a stronger signature of the data in some applications (see Section~\ref{subsec:application}).

\subsection{Harmonic and essential}
\label{subsec:essential}
We establish an important connection between the harmonicity of a representative
cycle and the set of essential simplices of a bar.
If a bar in the barcode of a filtration
is of multiplicity one (this happens generically), then it is represented by a unique harmonic 
representative (unique up to multiplication by non-zero scalar). 
We define for each cycle,  
\[
z = \sum_{\sigma} c_\sigma \cdot \sigma,
\]
(not necessarily harmonic) representing any given simple bar the relative
essential content,
\[
\econt(z) = \left(\frac{\sum_{\sigma \mbox{ is essential}}  c_\sigma^2}{\sum_{\sigma} c_\sigma^2}\right)^{1/2},
\]
of the cycle (see Definition~\ref{def:content} for the precise definition)
which measures the relative weight in the cycle of the essential as opposed to the 
non-essential simplices.

\subsubsection{Our Result}
\label{subsubsec:essential}
We prove  (Theorem~\ref{thm:essential}) that the harmonic 
representatives of bars maximize (amongst all  representative cycles)
the relative essential content of the bar (see Definition~\ref{def:content}), 
i.e.
if $z_0$ is a harmonic representative of a simple bar $b$,
then for any cycle $z$  representing $b$,
\[
\econt(z) \leq \econt(z_0).
\]

This indicates that in applications where one would like to emphasize the role of essential simplices,  harmonic representatives of bars are preferable over (say) volume optimal ones mentioned above.

\subsection{Example}
\label{subsec:Example}
Before proceeding further we discuss an example which illustrates the notion
of harmonic represesentatives and essential simplices. 
A bar $b$ in the barcode of a filtration is described by a triple $(s,t,\mu)$,
where $s$ denotes the birth time, $t$ the death time and $\mu$ the multiplicity 
(see Definition~\ref{def:barcode2}).

We denote by
$\mathcal{P}^{s,t}_p(\FF)$ the $p$-dimensional persistent harmonic homology subspace of the chain space
$C_p(K)$
corresponding to the bar born at time $s$ and which dies at time $t$ (or never dies if $t = \infty$).

We denote by $\Sigma(b)$ the set of essential simplices associated to a simple bar $b$.

\begin{example}
\label{eg:essential}

 \begin{figure}
    \centering
    \includegraphics[scale=0.65]{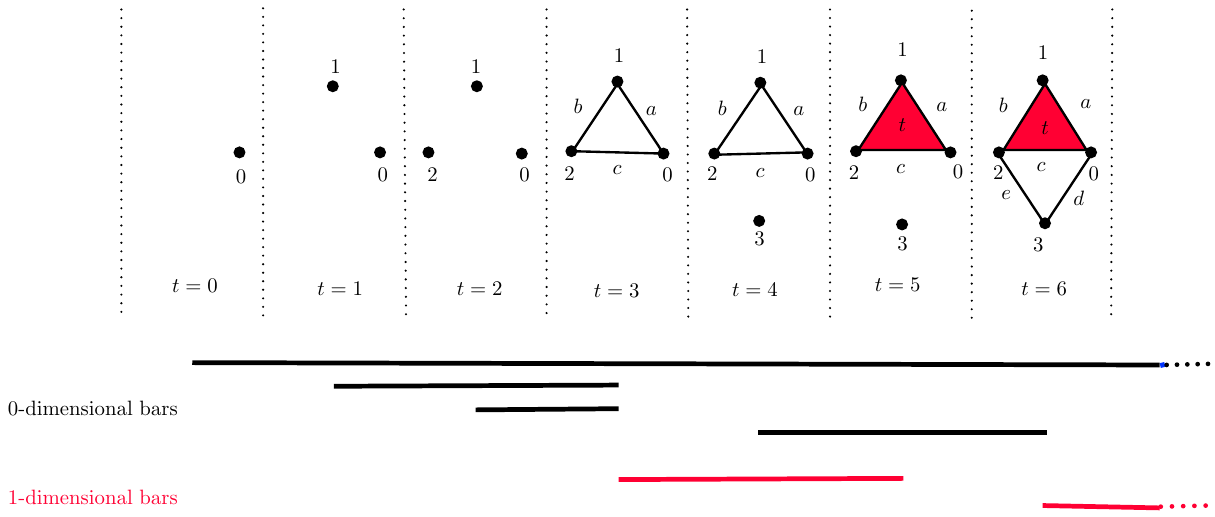}
    \caption{Barcode of a filtration}
    \label{fig:essential}
\end{figure}
Let $K$ be the simplicial complex defined by:
\begin{eqnarray*}
K^{\sqbracket{0}} &=& \{[0],[1],[2],[3] \}, \\
K^{\sqbracket{1}} &=& \{a =[0,1], b= [1,2], c= [0,2], d = [0,3], e= [2,3] \}, \\
K^{\sqbracket{2}} &=& \{ t = [0,1,2] \}.
\end{eqnarray*}
For $p=0,1,2$, we choose the standard inner product on $C_p(K)$ (see \eqref{eqn:standard}).

Let $\FF$ be the following filtration on the $K$:
\[
\emptyset\subset \{0\}\subset\{0,1\}\subset\{0,1,2\}\subset \{0,1,2,a,b,c\}\subset \{0,1,2,3,a,b,c\}
\]
\[
\subset \{0,1,2,3,a,b,c,t\}\subset \{0,1,2,3,a,b,c,d,e,t\}.
\]

For simplicity, we assume that vertex $\{0\}$ is added at time 0, and each complex in the filtration occurs at time 1 greater than the complex preceding it (see Figure~\ref{fig:essential}).
It is clear that all of the bars of the barcode of this filtration are simple.
The corresponding harmonic persistent homology 
subspaces  are listed in the following table. Note that since all the bars are simple, all these
subspaces have dimension $1$.
\begin{center}
\begin{tabular}{|c|c|}
\hline
    \multicolumn{2}{|c|}{$p=0$} \\
    \hline
     $\mathcal{P}_0^{0,\infty}(\mathcal{F})$& $\spn\{0\}$ \\
     \hline
     $\mathcal{P}_0^{1,3}(\FF)$ & $\spn\{1\}$ \\
     \hline
     $\mathcal{P}_0^{2,3}(\FF)$ & $\spn\{2\}$ \\
     \hline
     $\mathcal{P}_0^{4,6}(\FF)$ & $\spn\{3\}$ \\
     \hline
\end{tabular}
\end{center}
\begin{center}
\begin{tabular}{|c|c|}
\hline
    \multicolumn{2}{|c|}{$p=1$} \\
    \hline
     $\mathcal{P}_1^{3,5}(\mathcal{F})$& $\spn\{a+b-c\}$ \\
     \hline
     $\mathcal{P}_1^{6,\infty}(\FF)$ & $\spn\left\{a+b+2c-3d+3e\right\}$ \\
     \hline
\end{tabular}
\end{center}

For $p = 1$, the set of essential  simplices for each bar is listed below.
\begin{eqnarray*}
\Sigma((3,5,1)) &=& \{a,b,c\}, \\
\Sigma((6,\infty,1)) &=& \{d,e\}.
\end{eqnarray*}

The relative essential content of the harmonic representatives of these two bars are given by
\begin{eqnarray*}
\econt((3,5,1)) &=& 1,\\
\econt((6,\infty,1)) &=& \left(\frac{3}{4}\right)^{1/2}.
\end{eqnarray*}

Consider the simple bar born at time $6$. The homology class
corresponding to this bar can be represented by many (infinitely)
different cycles. The shortest or volume optimal cycle  representing it is $c + d - e$. The harmonic representative is given by
$a + b + 2c - 3d + 3e$.

The relative essential content of the volume optimal cycle is 
$\left(\frac{2}{3}\right)^{1/2}$ which is strictly smaller than that of the harmonic representative which is equal to $\left(\frac{3}{4}\right)^{1/2}$.
\end{example}

\begin{remark}
\label{rem:kurlin}
Note that
in \cite{kurlin2015one,Kurlin2019} the authors define 
the notion of critical simplices
that is related but is not equivalent to the notion of being essential
as defined in this paper. In fact it is easy to come up with examples
of filtrations where the set of ``critical simplices'' in the sense of \cite{kurlin2015one} is properly contained in the set of essential simplices.
\hide{
For example, if one considers the filtration of $1$-dimensional complexes
$
K_0 \subset K_1 \subset \cdots \subset K_n,
$
where for $0 \leq i < n$, 
$
K_i = \{[0,1],\ldots, [0,i+1]\},
$
and 
$
K_n = K_{n-1} \cup \{[n,0]\},
$
the set of critical edges contain just the edge $[n,0]$ which completes the cycle in the last step of the filtration, but every edge is essential for the bar in dimension one represented by the triple
$(n,\infty,1)$.
}
\end{remark}

\subsection{Stability of harmonic persistent homology and harmonic barcodes}
Stability theorems say that the persistent homology or its associated barcode
is stable under perturbations of the input data. This makes persistent homology useful in applications, where the data often comes from physical measurements with their attendant sources of error.
A basic body of results that makes persistent homology theory relevant in topological data analysis
is about the stability of the persistence module \cite{Cohen-Steiner-et-al-2007, Cohen-Steiner-et-al-2009, Cohen-Steiner-et-al-2010}.
There are many variations of stability results in the literature. 
We refer the reader to \cite[Chapter VIII]{Edelsbrunner-Harer2010} and 
\cite[\S 5.6]{Chazal-et-al} for a survey of these results.

The Euclidean space structure on the space of chains gives us the 
ability to talk about distances between harmonic homology subspaces 
corresponding to
different sub-complexes 
of some fixed ambient simplicial complex. This is important because it allows us to prove stability 
theorems -- \emph{sub-complexes which are close should have harmonic homology spaces which are close 
under a natural metric}  (see Theorem~\ref{thm:stability-homology} and Example~\ref{eg:ladder}). 

In order to discuss distance between harmonic homology subspaces, we need
a notion of distance between subspaces of a finite dimensional real vector space $V$. The set of all $d$-dimensional subspaces of a vector space $V$ is a well-studied topological space (in fact, has the structure of a projective algebraic variety) called the Grassmannian denoted $\Gr(d,V)$. Since, we will have subspaces of possibly different dimensions, we need a metric not just on $\Gr(d,V)$ but    
on the disjoint union 
\[
\Gr(V) = \coprod_{0\leq d \leq \dim V} \Gr(d,V)
\]
where $V$ is an Euclidean space $V$ ($V$ will be a
chain space equipped with an inner product).

Metrics on
Grassmannians were studied in detail by Lim and Ye  \cite{Lim-Ye2016} where the authors introduce several notions of distance between subspaces of varying dimensions. 
We use in this paper the distance called ``Grassmann distance'' in \cite{Lim-Ye2016}. This distance is defined in terms of the 
``principal angles'' between two subspaces and is defined precisely later (see Definition~\ref{def:grassmetric}). 
We prove a new class of stability theorems -- bounding the distance between two 
harmonic filtration functions (see Definition~\ref{def:persfunc}) in terms of certain norms of the difference of the functions inducing the filtrations
(see Theorems~\ref{thm:stable-persistent} and \ref{thm:stable}).
These theorems should be compared with corresponding results for classical persistent homology
groups (see \cite{Edelsbrunner-Harer2010,Chazal-et-al}). 

\subsubsection{Our results}
We prove several different stability theorems. 

Firstly, we  leverage the fact that 
harmonic homology spaces 
(we denote the harmonic homology subspace of dimension $p$ of a simplicial complex $K$ by $\mathcal{H}_p(K)$) are elements
of the Grassmannian $\Gr(C_p(K))$, where $C_p(K)$ is the $p$-th chain group of 
the simplicial complex $K$.
The Grassmannian $\Gr(C_p(K))$ carries a natural metric
induced by the Euclidean inner product on $C_p(K)$). Thus, it is meaningful to ask for a stability
result for the subspaces of $\mathcal{H}_p(K)$ themselves. We prove 
(see Theorem~\ref{thm:stability-homology} in Section~\ref{sec:harmonic-homology}) that
for $K$ a finite simplicial complex and $K_1,K_2$ sub-complexes of $K$,
for each $p \geq 0$,
$
d_{K,p}(K_1,K_2) \leq   \frac{\pi}{2}\cdot \Delta_p(K_1,K_2)^{1/2}    
$
where $d_{K,p}(K_1,K_2)$ denotes the Grassmannian 
distance between $\mathcal{H}_p(K_1)$ and
$\mathcal{H}_p(K_2)$ in the Grassmannian $\Gr(C_p(K))$
(Definition~\ref{def:grassmetric}), and 
$\Delta_p(K_1,K_2)$ is a natural measure of difference between $K_1$ and $K_2$ 
which is defined precisely in Theorem~\ref{thm:stability-homology}. \\

We next consider the harmonic persistent homology of filtrations of a finite simplicial complex $K$ induced by certain admissible functions $f:K \rightarrow \R$ (see Definition~\ref{def:simplex-wise}). 
Such a function induces for each $p \geq 0$, two functions,
$
 t \mapsto \mathcal{H}_p(K_{f \leq t}), 
$
$
(s,t) \mapsto \mathcal{H}_p^{s,t}(\mathcal{F}_f)
$.
Here $K_{f \leq t}$ denotes the sub-level complex of $f$ (see Notation~\ref{not:induced-filtration}), and
$\mathcal{H}^{s,t}_p(\mathcal{F}_f)$ is the $(p,s,t)$-th harmonic persistent 
homology subspace of the filtration $\mathcal{F}_f$ induced by $f$
(see Definition~\ref{def:harmonic-persistent}).

A stability result on each of the above functions should state that
for any pair of admissible functions $f,g:K \rightarrow \R$ which are close (under some metric)
the corresponding pairs of functions
$
t \mapsto \mathcal{H}_p(K_{f \leq t}), t \mapsto \mathcal{H}_p(K_{g \leq t})
$
as well as 
$
(s,t) \mapsto \mathcal{H}_p^{s,t}(\mathcal{F}_f), (s,t) \mapsto \mathcal{H}_p^{s,t}(\mathcal{F}_g)
$,
should be 
close to each other.
We prove that (Theorems~\ref{thm:stable} and \ref{thm:stable-persistent}), given two 
functions $f,g:K \rightarrow \R$ (satisfying a certain technical condition),
the distance between the corresponding function pair
$
t \mapsto \mathcal{H}_p(K_{f \leq t}), t \mapsto \mathcal{H}_p(K_{g \leq t})
$
as well as the pair
$(s,t) \mapsto \mathcal{H}_p^{s,t}(\mathcal{F}_f), (s,t) \mapsto \mathcal{H}_p^{s,t}(\mathcal{F}_g),
$
(defined via integrating over the appropriate domains the respective Grassmannian distances),
is bounded from above by a constant times certain semi-norms of the functions of the function $f-g$
(see Definition~\ref{def:norm-on-functions}). \\

We also study the stability of the harmonic barcodes. 
Let $K$ be a finite simplicial complex and  
let $\mathcal{F}$ denote a finite filtration
$K_0 \subset \cdots \subset K_N= K$.
The harmonic barcode
subspaces $\mathcal{P}^{s,t}_p(\mathcal{F})$  (see Definition~\ref{def:harmonic-barcode2}) 
are subspaces of $\mathcal{H}_p(K_s)$ (corresponding to the birth of 
the homology classes corresponding to the harmonic bar
$(s,t,\mathcal{P}^{s,t}_p(\mathcal{F}))$ (assuming
$\mathcal{P}^{s,t}_p(\mathcal{F}) \neq 0$). 
It also makes sense to consider the subspace of $\mathcal{H}_p(K_{t-1})$
representing the same bar  just before its  death -- which we denote by
$\widehat{\mathcal{P}}^{s,t}_p(\mathcal{F})$ (see Definition~\ref{def:terminal}).
We call $\widehat{\mathcal{P}}^{s,t}_p(\mathcal{F})$ the terminal harmonic subspace representing the 
corresponding bar.
For technical reasons we prove stability of the terminal 
harmonic subspaces representing ``generic'' bars (precise definition of ``generic'' appears later
in Definition~\ref{def:generic}).
We first define an appropriate notion of distance between harmonic barcodes of two
different filtrations
satisfying a genericity assumption.
The distance measured introduced for proving stability of the harmonic
homology subspaces and also the harmonic persistent homology subspaces 
(Theorems~\ref{thm:stable} and \ref{thm:stable-persistent})  
are in the form of an integral
(see Definitions~\ref{def:metric} and \ref{def:persistent-harmonicdist}).
Since for a filtration
$\mathcal{F}$ of a finite simplicial complex $K$, the subspaces
$\widehat{\mathcal{P}}^{s,t}_p(\mathcal{F})$ will be non-zero only for a finitely
many pairs $(s,t)$, the integral form of the distance function is not suitable.

For two finite filtrations $\mathcal{F}$ and $\mathcal{G}$ indexed by the same ordered indexing set of cardinality $N+1$, 
we will use the sum
(see Theorem~\ref{thm:stable-persistent-terminal} below)
$
\frac{1}{\binom{N+1}{2}}\cdot \sum_{s < t }  
d_{C_p(K)}(\widehat{\mathcal{P}}_p^{s,t}(\mathcal{F}), 
\widehat{\mathcal{P}}_p^{s,t}(\mathcal{G}))^2
$
as a measure of distance between the harmonic barcodes of $\mathcal{F},\mathcal{G}$,
where $d_{C_p(K)}(\cdot,\cdot)$
denotes the Grassmann distance between subspaces in $C_p(K)$ mentioned earlier
(see Definition~\ref{def:grassmetric} for the precise definition of $\dist_{C_p(K)}(\cdot,\cdot)$).

We prove that (Theorem~\ref{thm:stable-persistent-terminal}) 
\[
\frac{1}{\binom{N+1}{2}}\cdot  \sum_{0 \leq s < t \leq 1}  
d_{C_p(K)}(\widehat{\mathcal{P}}_p^{s,t}(\mathcal{F}), 
\widehat{\mathcal{P}}_p^{s,t}(\mathcal{G}))^2
\]
is bounded from above
by 
\[
\frac{\pi^3}{2} \cdot \left( ||f-g||_1^{(p)} + ||f-g||_1^{(p+1)}\right)^2
\]
where $f,g:K \rightarrow [0,1]$ 
are admissible functions inducing the filtrations $\mathcal{F},\mathcal{G}$, and the $1$-norms
$||f-g ||_1^{(p)}, || f-g ||_1^{(p+1)}$ measuring the difference between the two filtering functions
$f$ and $g$ are defined in Definition~\ref{def:norm-on-functions}.

The stability results described above 
give theoretical validity to the use of harmonicity in persistent homology theory.

\subsection{Prior and related work}
\label{sec:history}
The definition of harmonic subspaces of chain spaces of a simplicial complex goes back to the
work of Eckmann \cite{Eckmann}. It has been discussed in the context of statistical inference and developing
and studying graph Laplacians and Hodge theory on graphs by Lim in \cite{Lim2020}. 
The theory of harmonic homology is closely related to generalized Hodge theory and $L^2$-cohomology. 
The study of Hodge theory on general metric spaces with motivation coming from data analysis and computer vision
was initiated in Bartholdi et al.\  in \cite{Bartholdi-et-al-2012}. However, the emphasis in their work  
is different from that of the current paper and is concerned more about the connections and interplay between 
the generalized and the classical 
Hodge theory on Riemannian manifolds. The theory of 
$L^2$-cohomology of finite simplicial complexes was studied from an algorithmic
perspective by Friedman in \cite{Friedman96} 
who gave an efficient algorithm for computing them. 
Persistent harmonic cohomology has also being mentioned before \cite{talk:Lieutier} (see also \cite{Chen-et-al-2021}),
but the 
emphasis is more on manifold-learning rather than on general simplicial filtrations considered 
in this paper. Memoli et al.\  \cite{Memoli-et-al} also studies persistent homology groups,
defining them in terms of the Laplace operator (see Remark~\ref{rem:harmonic}) 
and gives efficient
algorithms for computing them. They also establish interesting connections with spectral graph theory and
prove certain stability results on the eigenvalues of the Laplace operator when applied to a simplicial
filtration.
These results and those in the current paper are somewhat orthogonal and
it would be interesting to investigate if they are related. 
The definition of harmonic barcodes of filtrations defined in the current paper 
is  new to our knowledge. \\

\subsection{Applications}
\label{subsec:application}
Persistent homology barcodes have found wide applications in many different areas.
Persistent harmonic barcodes as defined in this paper carry more information  but
has not yet been applied in practice. In this section we discuss possible applications
where the extra information in the persistent harmonic barcodes could potentially prove useful.

The data to which persistent homology methods are applied can be broadly categorized 
into two classes -- labelled and unlabelled. The typical example of unlabelled data 
are point cloud data approximating some underlying manifold. Persistent homology is
a tool to understand the global topology of the manifold. The individual points 
in the point cloud are not so significant by themselves -- it is only their interaction with neighbors that is important. 
If the point cloud is contained in some Euclidean space,
then the persistent homology barcode of the Vietoris-Rips filtrations is invariant under
any isometry applied to the data.

The situation is quite different if the input data is labelled. In this case each labelled  (weighted) simplex
in the corresponding simplicial complex carry information about some relationship
between its vertices. The situation is now more rigid and one needs to be able to distinguish between two isomorphic simplicial filtrations related (say) by a permutation 
of its vertices. The ordinary persistent homology barcodes cannot distinguish between two such
filtrations -- however, the persistent harmonic barcodes introduced in this paper
with the extra information will be able to distinguish between them.
This is very important in applications (such as in genomics) where
persistent homology methods are applied to ``relationship'' data as opposed to point-cloud data 
from some $\mathbb{R}^d$ (see 
\cite{talk:Parida} and also  \cite{Mandal-et-al-2020, Guzman-et-al-2019, Lesnick-et-al-2020, Rabadan-et-al-2019} for recent representative examples of such applications)

In practical applications, 
often the barcodes obtained from the input data are fed into some standard machine learning algorithm
such as a convolution neural net (CNN). 
Since the harmonic barcodes carry more refined information and have desirable properties
(stability and maximizing relative essential content) it is reasonable to postulate
that in a two-step pipeline feeding extra information regarding the harmonic persistent barcodes to the 
machine learning programs could improve the quality of the output.

We now describe a specific scenario where harmonic persistent barcodes could prove useful. This is an example of an application of topological data analysis in genomics where the labels of vertices are significant.

\subsubsection{Phenotype prediction}
In \cite{Mandal-et-al-2020} the authors investigate 
if gene expression measured from RNA sequencing contains enough signal to separate healthy from individuals
afflicted with Parkinson's disease. 
Topological data analysis (persistence barcode) was used in conjunction with
a certain standard machine learning tool (CNN) and the approach
 yielded improved results on Parkinson’s disease phenotype prediction when measured against standard machine learning methods.
The results in the paper thus  confirm that gene expression can be a useful indicator of the presence or absence of a condition, and the subtle signal contained in this high dimensional data reveals itself when considering the intricate topological connections between expressed genes.

The input is  sequencing-based gene expression values from blood samples. More precisely,
the input is a matrix $X$ of size whose rows correspond to subjects and whose columns  correspond to genes. Each entry $X_{i,j}$  
is the expression value of the $j$-th gene expression of the $i$-th subject.

Using a distance correlation matrix obtained from $X$ (see Section 2.3 in \cite{Mandal-et-al-2020}) for each individual
and every pair of genes a value is computed which measures the pair-wise interaction of these two genes --
and from these values one creates a Vietoris-Rips filtration of a simplicial comlpex (whose
zero dimensional simplices are labelled by the genes).
Using persistent homology computation one obtains for each individual their persistent homology barcodes.
Actually, what is computed 
are persistence landscapes (this is a variant of barcode introduced in \cite{Bubenik} and is particularly suitable to give at input to machine learning algorithm used in the second step of the pipeline).
These landscapes are then used to train a convolution neural network to separate the healthy from the 
afflicted
(see Figure 2 in \cite{Mandal-et-al-2020} reproduced as Figure~\ref{fig:Mandal-et-al}).
\begin{figure}
    \centering
    \includegraphics[scale=0.50]{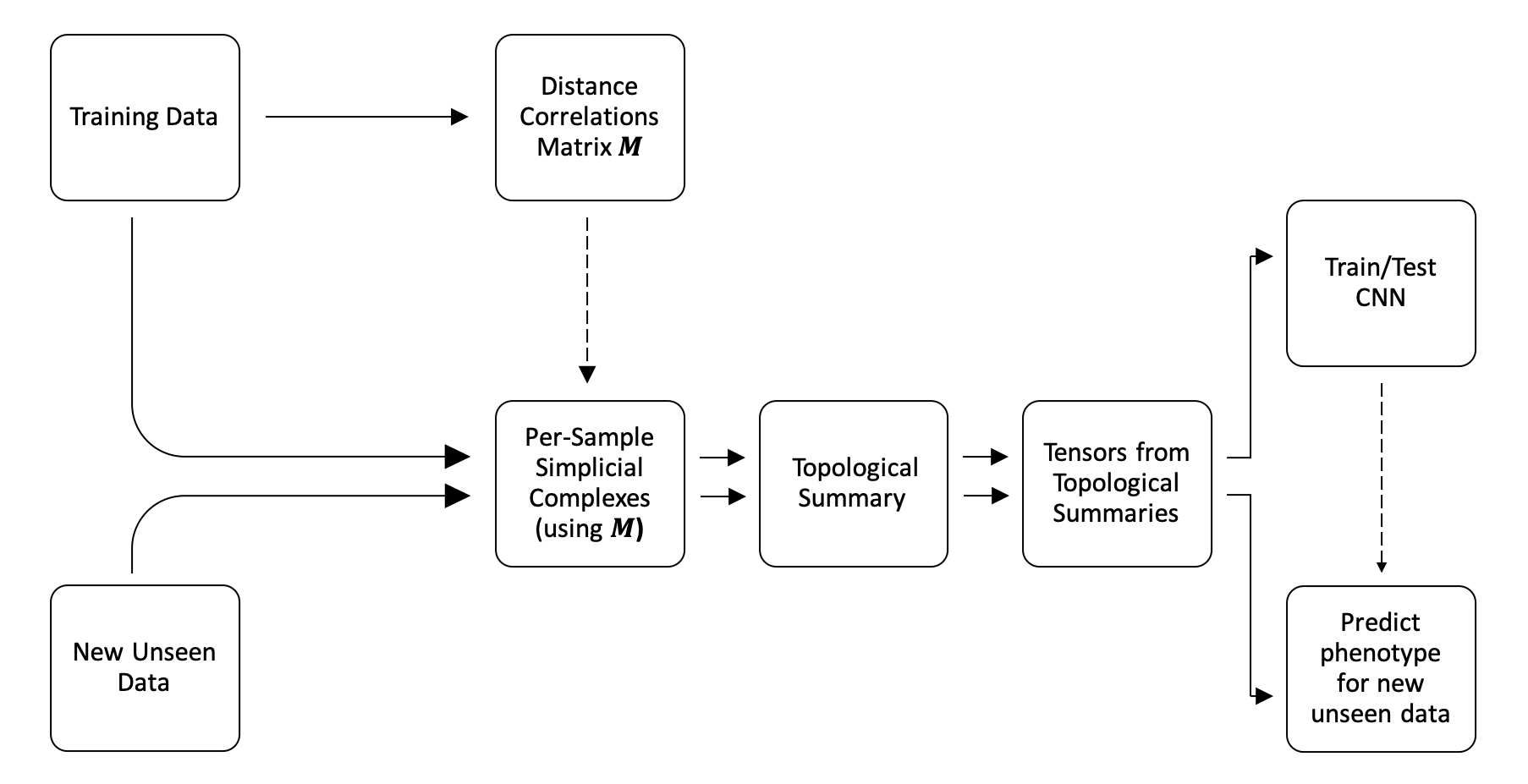}
    \caption{Pipeline reproduced from  \cite{Mandal-et-al-2020}}
    \label{fig:Mandal-et-al}
\end{figure}
This method of extracting the topological information first (i.e. computing the topological landscapes) and then using machine learning techniques was found to be more effective than
just using machine learning methods on the raw data (i.e. on the matrix $X$). 

The vertices of the Vietoris-Rips complex derived from the input data 
corresponds to genes
and thus the filtration considered  is a labelled filtration. 
We plan to augment the output of the persistent homology computation by adding to each bar
$b = (s,t,1)$
(assumed to be simple) unit vectors whose spans are the corresponding initial and 
terminal harmonic homology subspaces 
$\mathcal{P}^{s,t}_p(\mathcal{F}), \widehat{\mathcal{P}}^{s,t}_p(\mathcal{F})$ (Definition~\ref{def:terminal})
associated to $b$.
(Of course, we need to update the persistent homology software to compute these vectors and
develop algorithms for computing them efficiently.)
This should
increase the amount of discerning information in  the input to the ML algorithms 
in the second step of the  pipeline. We postulate that doing so would improve the power of the method.
The experimental results on whether such improvements are actually observed
in practice is under investigation and will be reported in a subsequent paper. \\

The rest of the paper is organized as follows. In Section~\ref{sec:harmonic-homology}, we give the necessary mathematical background and define harmonic homology spaces and prove their basic properties, including a key theorem (Theorem~\ref{thm:stability-homology}) bounding the distance between harmonic homology subspaces of two simplicial complexes in terms of the difference between the two simplicial  complexes.
In Section~\ref{sec:harmonic-persistent-homology}, we define harmonic persistent homology, harmonic bar codes and
prove three stability theorems (Theorems~\ref{thm:stable}, \ref{thm:stable-persistent}, \ref{thm:stable-persistent-terminal}).
Finally, in Section~\ref{sec:essential}, we 
prove that harmonic representatives maximize relative essential content (Theorem~\ref{thm:essential}).

\section{Harmonic homology}
\label{sec:harmonic-homology}
In this section we define harmonic homology spaces of simplicial complexes and prove their basic properties.

\subsection{Simplicial complex}
We  recall here some basic definitions and notation from simplicial homology theory.
\begin{definition}
\label{def:simplicial-complex}
A \emph{finite simplicial complex} $K$ is a set of ordered subsets of $[N] = \{0,\ldots N\}$ for some $N \geq 0$, such that if $\sigma \in K$ and $\tau$
is a subset of $\sigma$, then $\tau \in K$. 
\end{definition}

\begin{notation}
\label{def:p-simplices}
If $\sigma = \{i_0,\ldots,i_p\} \in K$, 
with $K$ a finite simplicial complex,
and $i_0 < \cdots < i_p $, we will denote $\sigma = [i_0,\ldots,i_p]$ and call $\sigma$ a \emph{$p$-dimensional simplex of $K$}. 
We will denote by $K^{(p)}$ the subcomplex of $K$ consisting of simplices of $K$ of dimension $\leq p$.
We will denote by $K^{\sqbracket{p}} = K^{(p)} \setminus K^{(p-1)}$ the subset of $p$-dimensional simplices of $K$.
Note that $K^{\sqbracket{p}}$ is not a subcomplex.
\end{notation}

\begin{definition}[Chain groups]
\label{def:chain-groups}
Suppose $K$ is a finite simplicial complex.
For $p \geq 0$,  we will denote by 
$C_p(K) = C_p(K,\R)$ (the $p$-th chain group),   
the $\R$-vector space generated by the elements of $K^{\sqbracket{p}}$,
i.e.
\[
C_p(K) = \bigoplus_{\sigma \in K^{\sqbracket{p}}}\R \cdot \sigma.
\]
\hide{
The tuple
$\left(\sigma \right)_{\sigma \in K^{\sqbracket{p}}}$ is then a basis (called the standard basis) of $C_p(K)$
(where by a standard abuse of notation we identify $\sigma$ with the image of $1\cdot \sigma$ under the canonical injection of the direct summand $\R \cdot \sigma$ into $C_p(K) = \bigoplus_{\sigma \in K^{\sqbracket{p}}}\R \cdot \sigma$).
}
\end{definition}
\begin{definition}[The boundary map]
\label{def:boundary-map}
We denote by $\partial_p(K): C_p(K) \rightarrow C_{p-1}(K)$ the 
linear map (called the $p$-th \emph{boundary map})  defined as follows. Since 
$\left(\sigma \right)_{\sigma \in K^{\sqbracket{p}}}$ is a basis of $C_p(K)$ it is enough to define the image of each $\sigma \in C_p(K)$. 
We define for $\sigma = [i_0,\ldots,i_p] \in K^{\sqbracket{p}}$,
$
\partial_p(K)(\sigma) = \sum_{0\leq j \leq p} (-)^j [i_0,\ldots, \widehat{i_j}, \ldots,i_p] \in C_{p-1}(K),
$
where $\widehat{\cdot}$ denotes omission.
\end{definition}
\hide{
One can easily check that the boundary maps $\partial_p$ satisfy the key property that 
\[
\partial_{p+1}(K) \;\circ\; \partial_{p}(K) = 0,
\]
or equivalently
that 
\[
\mathrm{Im}(\partial_{p+1}(K)) \subset \Ker(\partial_p(K)).
\]
}
\begin{notation}[Cycles, boundaries, homology and the canonical surjection]
We denote 
$
Z_p(K) = \Ker(\partial_p(K)),
$
(the space of \emph{$p$-dimensional cycles}),
$
B_p(K) = \mathrm{Im}(\partial_{p+1}(K))
$
(the space of \emph{$p$-dimensional boundaries}),
and
$\HH_p(K) = Z_p(K)/B_p(K)$
(the \emph{$p$-dimensional simplicial homology group} of $K$).
We will denote by 
\[
\phi_p(K):Z_p(K) \rightarrow Z_p(K)/B_p(K) =  \HH_p(K)
\]
the canonical surjection.
\end{notation}

\subsection{Representing homology classes by harmonic chains} 
Let $K$ be a finite simplicial complex.
We make the chain group $C_p(K)$ into an Euclidean space by fixing an inner product $\la \cdot,\cdot\ra_{C_p(K)}$. 
For the rest of the paper we fix the following inner product on $C_p(K)$ which we will    
refer to as the standard inner product on $C_p(K)$.
We define:
\begin{equation}
\label{eqn:standard}
\la \sigma, \sigma'\ra_{C_p(K)} = \delta_{\sigma,\sigma'},  \sigma,   \sigma' \in K^{\sqbracket{p}}
\end{equation}
(i.e. we declare the basis  $\left(\sigma \right)_{\sigma \in K^{\sqbracket{p}}}$ to be an orthonormal basis).
If the context is clear we will omit the subscript from the notation  $\la \cdot,\cdot\ra_{C_p(K)}$.

We now come to a key definition -- namely, that of harmonic homology (following \cite{Eckmann}).
\begin{definition}[Harmonic homology subspace]
\label{def:harmonic}
For $p \geq 0$,  we will denote 
\[
\mathcal{H}_p(K) = Z_p(K)  \cap  B_p(K)^\perp
\]
and call $\mathcal{H}_p(K) \subset C_p(K)$ the \emph{harmonic homology subspace of $K$}. 
\end{definition}

\subsubsection{Elementary properties}
The following two propositions encapsulate the key properties of the harmonic homology subspaces.
\begin{proposition}
\label{prop:f}
The map  $\mathfrak{f}_p(K)$ defined by 
\begin{equation}
    \label{eqn:f}
    z + B_p(K) \rightarrow \mathrm{proj}_{B_p(K)^\perp}(z), z \in Z_p(K)
\end{equation}
gives an isomorphism
$
\mathfrak{f}_p(K): \HH_p(K)  \rightarrow \mathcal{H}_p(K).
$
\end{proposition}

\begin{proof}
First observe that 
using 
the fact that $B_p(K) \subset Z_p(K)$,
we have that for $z \in Z_p(K)$, $\mathrm{proj}_{B_p(K)^\perp}(z) \in Z_p(K)$, and so the
map $\mathfrak{f}_p(K)$ is well defined. The injectivity and surjectivity of $\mathfrak{f}_p(K)$ are
then obvious.
\end{proof}

\begin{proposition}
\label{prop:intersection}
$
\mathcal{H}_p(K) = \Ker(\partial_{p+1}(K)^*) \cap \Ker(\partial_{p}(K))
$
(where $L^*$ denotes the adjoint of a linear map $L$ between two inner product spaces).
\end{proposition}

\begin{proof}
From Definition~\ref{def:harmonic} and the fact that 
$Z_p(K) = \Ker(\partial_p)$, it 
suffices to prove that 
\[\Ker(\partial_{p+1}(K)^*) = B_p(K)^\perp.
\]
For $z \in C_p(K)$,
\begin{eqnarray*}
z \in B_p(K)^\perp &\Leftrightarrow& \la z, z'\ra_{C_p(K)} = 0 \mbox{ for all } z' \in B_p(K) \\
&\Leftrightarrow& \la z, \partial_{p+1}(w) \ra_{C_p(K)} = 0 \mbox{ for all } w \in C_{p+1}(K) \\
&\Leftrightarrow& \la \partial_{p+1}^*(z), w\ra_{C_{p+1}(K)} = 0 \mbox{ for all } w \in C_{p+1}(K) \\
&\Leftrightarrow& z \in \Ker(\partial_{p+1}^*).
\end{eqnarray*}
This completes the proof of the proposition.
\end{proof}

\begin{remark}
\label{rem:harmonic}
The harmonic homology group
$\mathcal{H}_p(K)$ as defined above is
equal to the kernel of the linear map 
$\Delta_p = \partial_{p+1}\circ \partial_{p+1}^* + \partial_p^*\circ \partial_p$.
The linear map $\Delta_p(K):C_p(K) \rightarrow C_p(K)$ is a discrete analog of the Laplace operator and 
thus it makes sense to call its kernel the space of harmonic cycles.
\end{remark}

Since the above description is often taken as a definition of harmonic homology groups
we include the proof of the equivalence of the two definitions below.
\begin{proposition}
\label{prop:Laplacian}
$
\Ker (\Delta_p(K)) = \mathcal{H}_p(K)
$.
\end{proposition}

\begin{proof}
It follows directly from Proposition~\ref{prop:intersection} and the
definition of $\Delta_p(K)$ that 
$
\mathcal{H}_p(K) \subset \Ker(\Delta_p(K)).
$

In order to prove the opposite inclusion. Suppose $z \in \Ker(\Delta_p(K))$.
Then, 
$
\partial_{p+1}\circ\partial_{p+1}^*(z) +\partial_p^*\circ \partial_p(z) = 0
$.

Taking inner product with $z$ we obtain
\[
\la z, \partial_{p+1}\circ\partial_{p+1}^*(z)\ra_{C_p(K)} + \la z, \partial_p^*\circ\partial_p(z) \ra_{C_p(K)} = 0,
\]
which using the defining property of adjoints gives
\[
\la \partial_{p+1}^*(z), \partial_{p+1}^*(z)\ra_{C_{p+1}(K)} + \la \partial_p(z), \partial_p(z) \ra_{C_{p-1}(K)} = 0.
\]

This implies that 
$
z \in \Ker(\partial_{p+1}^*) \cap \Ker(\partial_p) = \mathcal{H}_p(K),
$
the last equality being a consequence of Proposition~\ref{prop:intersection}.
\end{proof}

\subsubsection{Functoriality of the maps $\mathfrak{f}_p(K)$ under inclusion}
Now suppose $K_1 \subset K_2$ are sub-complexes of the finite simplicial complex $K$. 
Then, $C_p(K_1)$ is a subspace of $C_p(K_2)$. 

\begin{proposition}
\label{prop:functorial}
The restriction of $\mathrm{proj}_{B_p(K_2)^\perp}$ to $\mathcal{H}_p(K_1)$ gives a linear map
\[
\mathfrak{i}_p= \mathrm{proj}_{B_p(K_2)^\perp}|_{\mathcal{H}_p(K_1)} : \mathcal{H}_p(K_1) \rightarrow \mathcal{H}_p(K_2),
\]
which makes the following diagram commute
\[
\xymatrix{
\HH_p(K_1) \ar[r]^{i_p} \ar[d]^{\mathfrak{f}_p(K_1)} & \HH_p(K_2) \ar[d]^{\mathfrak{f}_p(K_2)} \\
\mathcal{H}_p(K_1) \ar[r]^{\mathfrak{i}_p} &\mathcal{H}_p(K_2)
}
\]
where $i_p: \HH_p(K_1) \rightarrow \HH_p(K_2)$ is the map induced by the inclusion $K_1 \hookrightarrow K_2$.
\end{proposition}

Before proving the proposition we first prove a very basic lemma.

\begin{lemma}
\label{lem:functorial}
Let $B_1 \subset Z_1$, $B_2 \subset Z_2$ be subspaces of an Euclidean vector space $V$, and
suppose that $B_1 \subset B_2$, and $Z_1 \subset Z_2$. Then for $z_1 \in Z_1$,
\[
\proj_{B_2^\perp} \circ \proj_{B_1^\perp}(z_1) = \proj_{B_2^\perp}(z_1).
\]
\end{lemma}

\begin{proof}
First observe that $B_1 \subset B_2$ implies that $B_2^\perp 
\subset B_1^\perp$.

Now let $w_1 = \proj_{B_1^\perp}(z_1)$. Then $w_1 = z_1 - \proj_{B_1}(z_1)$.
\begin{eqnarray*}
\proj_{B_2^\perp} \circ \proj_{B_1^\perp}(z_1) &=& \proj_{B_2^\perp}(w_1) \\
&=& \proj_{B_2^\perp}(z_1) - \proj_{B_2^\perp}(\proj_{B_1}(z_1)) \\
&=& \proj_{B_2^\perp}(z_1),
\end{eqnarray*}
noting that
$
\proj_{B_2^\perp}(\proj_{B_1}(z_1)) = 0,
$
since (as noted earlier) $B_2^\perp \subset B_1^\perp$.
\end{proof}

\begin{proof}[Proof of Proposition~\ref{prop:functorial}]
Let $z_1 + B_p(K_1) \in \HH_p(K_1)$. Using the definition of $\mathfrak{f}_p(K_1)$ (see Eqn.\  \eqref{eqn:f} in Proposition~\ref{prop:f}) we have
\[
\mathfrak{f}_p(K_1)(z_1+ B_p(K_1)) = \mathrm{proj}_{B_p(K_1)^\perp}(z_1) \in \mathcal{H}_p(K_1) \subset 
B_p(K_1)^\perp.
\]

Now observe that 
$B_p(K_1) \subset Z_p(K_1)$, $B_p(K_2) \subset Z_p(K_2)$, $B_p(K_1) \subset B_p(K_2)$,
$Z_p(K_1) \subset Z_p(K_2)$.

So 
\begin{eqnarray*}
\mathfrak{i}_p(K) \circ \mathfrak{f}_p(K_1) (z_1 + B_p(K_1)) &=& 
\mathrm{proj}_{B_p(K_2)^\perp} \circ \mathrm{proj}_{B_p(K_1)^\perp}(z_1) \mbox{ (using \eqref{eqn:f})}\\
&=&
\mathrm{proj}_{B_p(K_2)^\perp}(z_1) \mbox{ (using Lemma~\ref{lem:functorial})}\\
&=& \mathfrak{f}_p(K_2) \circ i_p(z_1 + B_p(K_1)).
\end{eqnarray*}
This proves the proposition.
\end{proof}

\subsection{Stability of harmonic homology subspaces}
Suppose $K$ is a finite simplicial complex and $K_1,K_2$ sub-complexes of $K$. In various applications
one would want to compare the homology spaces of $K_1$ and $K_2$ quantitatively. In particular, one
wants to say that if $K_1$ and $K_2$ are close under some natural metric then so are $\HH_*(K_1)$ and
$\HH_*(K_2)$. Harmonic homology allows us to make rigorous this intuitive statement. 

Since for each $p \geq 0$, $\mathcal{H}_p(K_1), \mathcal{H}_p(K_2)$ will be subspaces of $C_p(K)$, and thus correspond to  points
in $\Gr(b_p(K_1),C_p(K)), \Gr(b_p(K_2),C_p(K))$ respectively,  
(denoting by $b_p(K_i) = \dim \mathcal{H}_p(K_i)$)
we first introduce a metric on the disjoint union of Grassmannians, 
$\coprod_{0 \leq d \leq \dim C_p(K)} \Gr(d,C_p(K))$.

\subsubsection{Metric on Grassmannian}
Let $V$ be an Euclidean space, and for $0 \leq d \leq \dim V$ we will denote by
$\Gr(d,V)$ the real Grassmannian variety of $d$-dimensional subspaces of $V$, and we will
denote 
\[
\Gr(V) = \coprod_{0 \leq d \leq \dim V} \Gr(d,V).
\]

Given two subspaces $A, B \subset V$, one way to measure how far away they are from each other
is to take the square root of the sum of the squares of the principal angles between $A$ and $B$. When the dimensions
of $A$ and $B$ are equal this works well and produces a metric on the Grassmannian $\Gr(d,V)$ where
$d = \dim A = \dim B$. However, if the dimensions of $A$ and $B$ are not necessarily equal, then this 
quantity can be zero even when $A \neq B$ (for example, if $A \subset B$). However, we would like
to distinguish between such subspaces in our applications.
In order to obtain
a metric on the disjoint union $\Gr(V)$, one needs to tweak the above definition.

Following  Lim and Ye \cite{Lim-Ye2016}, we now define a metric on $\Gr(V)$.
\begin{definition}
\label{def:grassmetric}
For $A \in \Gr(k,V), B \in \Gr(\ell,V)$ with $0 \leq k,\ell \leq \dim V$, we define
$$
d_{V}(A,B)=\left(|k-\ell|\frac{\pi^2}{4}+\ds\sum_{i=1}^{\min\{k,\ell\}} \theta_i^2\right)^{1/2}
$$
where $\theta_i$ is the $i$-th principal angle between $A$ and $B$.
\end{definition}

\subsubsection{Principal angles in terms of singular values}
The cosines of the principle angles between two subspaces of an Euclidean space $V$ can be expressed in terms
of the singular values of an associated matrix. Let $P,Q$ be two subspaces of $V$.
Fix an orthonormal basis $\mathbf{B}$ of $V$.
Suppose $\dim P = p \geq q = \dim Q$. Let $A,B$ be matrices whose
columns are the coordinates with respect to $\mathbf{B}$ of some orthonormal bases of $P$ and $Q$ respectively.  
Let
$U^T \Sigma V$ be the singular value decomposition of the $p \times q$ matrix $P^T Q$, with
$\Sigma$ a $p \times q$ matrix with diagonal entries the singular values $\sigma_1 \geq \cdots
\geq \sigma_q \geq 0$.

Then for $1 \leq i \leq q$,
\begin{equation}
\label{eqn:singular}
\sigma_i = \cos(\theta_i), 
\end{equation}
where $\theta_i$ is the $i$-th principle angle between
$P$ and $Q$.

\begin{lemma}
\label{lem:subspace}
Let $V$ be an Euclidean space and $W_1, W_2 \subset V$ be subspaces. 
Let 
\[
k = \max(\dim W_1,\dim W_2),
\]
and 
\[
\ell = \dim (W_1 \cap W_2).
\]
Then, 
\[
d_V(W_1,W_2) \leq  \frac{\pi}{2} \cdot (k - \ell)^{1/2}. 
\]
\end{lemma}

\begin{proof}
Let $k_1 = \dim W_1, k_2 = \dim W_2$.
It follows from  Definition~\ref{def:grassmetric}, and the definition of principle angles that
\begin{eqnarray*}
d_V(W_1,W_2) &\leq&  \frac{\pi}{2} \cdot \left( |k_1 - k_2| + \min(k_1,k_2) - \ell\right)^{1/2} \\
&=& \frac{\pi}{2} \cdot (k - \ell)^{1/2}.
\end{eqnarray*}
\end{proof}

We will also need the following lemma.
\begin{lemma}
\label{lem:singular}
Let $V$ be an Euclidean space and $W_1, W_2 \subset V$ be subspaces
having dimensions $p$ and $q$ respectively, with $p \geq q$. Let $L \subset W_2$ be a $1$-dimensional subspace.  
Let $\theta$ be the principal angle between $L$ and $W_1$, and $\theta_1 \leq \cdots \leq \theta_q$ be the principle angles between $W_1$ and $W_2$. Then, 
\[
\theta_1 \leq \theta \leq \theta_q.
\]
\end{lemma}

\begin{proof}
The first inequality follows from the variational characterization of the principle angles. We prove the 
second inequality. Let $\{e_1,\ldots,e_p\}$ be an orthonormal basis of
$W_1$.
Let $f = f_1$ be a unit length vector such that $\spanof(f) = L$, and let
$\{f_1,f_2,\ldots, f_q\}$ be an orthonormal basis of $W_2$.  Let $P$ and $Q$ be matrices whose columns are the coordinates of the $e_i$'s and the $f_i$'s respectively.
Let $Q'$ be the submatrix of $Q$ with consisting of the first column of $Q$ (i.e. the column of
coordinates of $f$).

Let $M = P^T Q, M' = P^T Q'$. Then, $M'$ is a $p \times 1$ submatrix of the $p \times q$ matrix $M$.
Let $\sigma$ be the singular value of $M'$, and $\sigma_1 \geq \cdots \geq \sigma_q$ the singular values of $M$.
It follows from Eqn.\ \eqref{eqn:singular} that 
\begin{eqnarray*}
\sigma &=& \cos \theta, \\
\sigma_i &=& \cos \theta_i, 1 \leq i \leq q.
\end{eqnarray*}
Finally it follows from 
the interlacing inequality  \cite[Theorem 10, Inequality (10)]{Thompson} that
$\sigma = \cos \theta  \geq \sigma_q = \cos \theta_q$, from which it follows that
$\theta \leq \theta_q$, since $0 \leq \theta,\theta_q \leq \pi/2$.
\end{proof}

\subsubsection{Stability theorem}
Using harmonic homology spaces we now define a distance function between the homology of two sub-complexes
of a fixed simplicial complex in any fixed dimension.

Let $K$ be a finite simplicial complex and $K_1,K_2$ be two sub-complexes. 

\begin{definition}
\label{def:distance-homology}
We define 
\[
d_{K,p}(K_1,K_2) =  d_{C_p(K)}(\mathcal{H}_p(K_1),\mathcal{H}_p(K_2)).
\]
\end{definition}

We are now in a position to make quantitative the intuitive idea that two sub-complexes of 
a fixed simplicial complex which are close to each other should have homology spaces that are also 
close. We prove the following theorem.

\begin{theorem}[Stability of harmonic homology]
\label{thm:stability-homology}
Let $K$ be a finite simplicial complex and $K_1,K_2$ sub-complexes of $K$. Then, for each $p \geq 0$,
\begin{equation}
\label{eqn:thm:stability-homology}
d_{K,p}(K_1,K_2) \leq   \frac{\pi}{2}\cdot \Delta_p(K_1,K_2)^{1/2}    
\end{equation}
where $\Delta_p(K_1,K_2)$ is defined to be the maximum of 
\[
\card\left(K_1^{\sqbracket{p}} \setminus   K_2^{\sqbracket{p}}\right) + \card\left(K_2^{\sqbracket{p+1}} \setminus   K_1^{\sqbracket{p+1}}\right),
\]
and
\[
\card\left(K_2^{\sqbracket{p}} \setminus  K_1^{\sqbracket{p}}\right) + \card\left(K_1^{\sqbracket{p+1}} \setminus K_2^{\sqbracket{p+1}}\right).
\]
\end{theorem}

Before proving Theorem~\ref{thm:stability-homology} we first prove a lemma that we will need
in the proof of Theorem~\ref{thm:stability-homology} and also later in the paper.

\begin{lemma}
\label{lem:stability-homology}
Let $K$ be a finite simplicial complex and $K_1,K_2$ sub-complexes of $K$. Then, for each $p \geq 0$,
\begin{equation}
\label{eqn:lem:stability-homology:1}
\dim \mathcal{H}_p(K_1) - \dim (\mathcal{H}_p(K_1) \cap \mathcal{H}_p(K_2)) \leq
\card\left(K_1^{\sqbracket{p}} \setminus  K_2^{\sqbracket{p}}\right) + \card\left(K_2^{\sqbracket{p+1}} \setminus K_1^{\sqbracket{p+1}}\right),
\end{equation}

\begin{equation}
\label{eqn:lem:stability-homology:2}
 \dim \mathcal{H}_p(K_2) - \dim (\mathcal{H}_p(K_1) \cap \mathcal{H}_p(K_2)) \leq
\card\left(K_2^{\sqbracket{p}} \setminus  K_1^{\sqbracket{p}}\right) + \card\left(K_1^{\sqbracket{p+1}} \setminus  K_2^{\sqbracket{p+1}}\right).
\end{equation}
\end{lemma}

\begin{proof}
We prove the inequality \eqref{eqn:lem:stability-homology:1}. The proof of inequality
\eqref{eqn:lem:stability-homology:2} is similar.

For each $p\geq 0$, and $i=1,2$, let $M_{p}(K_i)$ denote the matrix corresponding to the boundary map
$\partial_p: C_{p}(K_i) \rightarrow C_{p-1}(K_i)$ with respect to the orthonormal bases
$\mathcal{A}_p(K_i) = \left(\sigma\right)_{\sigma  \in K_i^{\sqbracket{p}}}$ of $C_p(K_i)$ and $
\mathcal{A}_{p-1}(K_i) = \left(\sigma\right)_{\sigma  \in K_i^{\sqbracket{p-1}}}$ 
of $C_{p-1}(K_i)$.

Note that the rows of $M_p(K_i)$ are indexed by $K_i^{\sqbracket{p-1}}$ and the columns of 
$M_p(K_i)$ are indexed by $K_i^{\sqbracket{p}}$.

Observe that 
$\mathcal{H}_p(K_1) \subset C_p(K_1)$ is the intersection of the nullspaces
of the matrices $M_{p}(K_1)$ and $M_{p+1}(K_1)^T$ i.e.
\[
z \in \mathcal{H}_p(K_1)  \Leftrightarrow [z]_{\mathcal{A}_p(K_1)}  \in \Ker(M_p(K_1)) \cap \Ker(M_{p+1}(K_1)^T).
\]

The subspace $\mathcal{H}_p(K_1) \cap \mathcal{H}_p(K_2)$ of $\mathcal{H}_p(K_1)$ is cut out 
of $\mathcal{H}_p(K_1)$ by additional equations.
Let $M'_p$ be the matrix whose columns are indexed by $K_1^{\sqbracket{p}}$ and whose rows are
indexed by $K_1^{\sqbracket{p-1}} \cup (K_1^{\sqbracket{p}} \setminus K_2^{\sqbracket{p}})$ defined by

\begin{eqnarray*}
(M'_p)_{\sigma,\sigma'} &=& M_p(K_1)_{\sigma,\sigma'} \mbox{ if $\sigma\in K_1^{\sqbracket{p-1}}, \sigma' \in K_1^{\sqbracket{p}}$,} \\
&=& 1, \mbox{ if $\sigma = \sigma'\in K_1^{\sqbracket{p}} \setminus  K_2^{\sqbracket{p}}$,} \\
&=& 0, \mbox{ otherwise.}
\end{eqnarray*}

Similarly, 
Let $M'_{p+1}$ be the matrix whose columns are indexed by $K_1^{\sqbracket{p+1}} \cup K_2^{\sqbracket{p+1}}$ and whose rows are
indexed by $K_1^{\sqbracket{p}}$ be defined by

\begin{eqnarray*}
(M'_{p+1})_{\sigma,\sigma'} &=& M_{p+1}(K_1)_{\sigma,\sigma'}, \mbox{ if $\sigma\in K_1^{\sqbracket{p}}, \sigma' \in K_1^{\sqbracket{p+1}}$,} \\
&=& M_{p+1}(K_2)_{\sigma,\sigma'}, \mbox{ if $\sigma\in K_1^{\sqbracket{p}}, \sigma' \in K_2^{\sqbracket{p+1}} \setminus   K_1^{\sqbracket{p+1}} $.} 
\end{eqnarray*}

Then,
\[
z \in \mathcal{H}_p(K_1) \cap  \mathcal{H}_p(K_2)  \Leftrightarrow [z]_{\mathcal{A}_p(K_1)}  \in \Ker(M_p') \cap \Ker(M_{p+1}'^T).
\]

Observe that $M'_p$ contains $M_p(K_1)$ as a submatrix and has $\card(K_1^{\sqbracket{p}} \setminus  K_2^{\sqbracket{p}})$ extra rows.
Similarly, $M_{p+1}'^T$ has $M_{p+1}(K_1)^T$ as a submatrix and has   
$\card(K_2^{\sqbracket{p+1}} \setminus   K_1^{\sqbracket{p+1}})$ extra rows.

It follows that the codimension of $\mathcal{H}_p(K_1) \cap  \mathcal{H}_p(K_2)$ in 
$\mathcal{H}_p(K_1)$ is bounded by 
\[
\card\left(K_1^{\sqbracket{p}} \setminus   K_2^{\sqbracket{p}}\right) + \card\left(K_2^{\sqbracket{p+1}} \setminus  K_1^{\sqbracket{p+1}}\right),
\]
which completes the proof of 
inequality \eqref{eqn:lem:stability-homology:1}.
\end{proof}

\begin{proof}[Proof of Theorem~\ref{thm:stability-homology}]
The theorem  now follows from Lemma~\ref{lem:subspace}, and inequalities \eqref{eqn:lem:stability-homology:1}, 
\eqref{eqn:lem:stability-homology:2} in Lemma~\ref{lem:stability-homology}.
\end{proof}

\begin{corollary}
\label{cor:stability-homology}
With the same assumptions as in Theorem~\ref{thm:stability-homology}:
\begin{equation}
\label{eqn:cor:stability-homology}   
d_{K,p}^2(K_1,K_2) \leq   \frac{\pi^2}{4}\cdot\left( \sum_{\sigma \in K^{\sqbracket{p}} \cup K^{\sqbracket{p+1}}} |\chi_{K_1}(\sigma) - \chi_{K_2}(\sigma)|        \right),
\end{equation}
where for any sub-complex $K'$ of $K$ we denote by 
$\chi_{K'}(\cdot)$ the characteristic function of  $K'$ (considered as a set of simplices). 
\end{corollary}
\begin{proof}
It is easy to see that the quantity $\Delta_p(K_1,K_2)$ defined in Theorem~\ref{thm:stability-homology} is bounded from above by 
\[
\sum_{\sigma \in K^{\sqbracket{p}} \cup K^{\sqbracket{p+1}}} |\chi_{K_1}(\sigma) - \chi_{K_2}(\sigma)|.
\]
The corollary 
is now an immediate consequence of Theorem~\ref{thm:stability-homology} after squaring both sides
of the inequality \eqref{eqn:thm:stability-homology}.
\end{proof}

\begin{remark}
\label{rem:crude}
The bound in  Theorem~\ref{thm:stability-homology} is a bit crude in that it only depends on the
dimension of the intersection of the two subspaces $\mathcal{H}_p(K_1)$ and $\mathcal{H}_p(K_2)$.
However, even if the dimension of the intersection stays constant, the principal angles between the
two subspaces give a measure of how ``close'' these two subspaces are. The following example
is instructive.
\end{remark}

\begin{example}
\label{eg:ladder}

\begin{figure}
    \centering
    \includegraphics[scale=0.9]{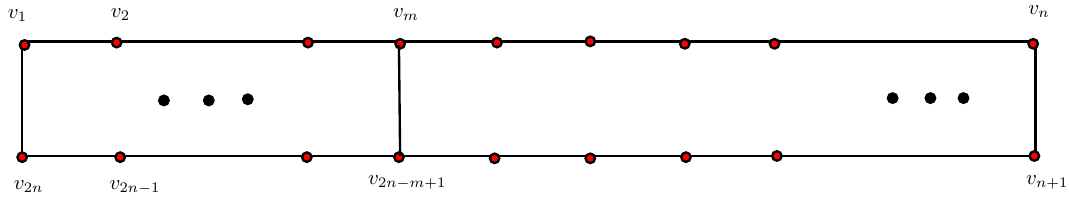}
    \caption{Convergence of harmonic homology subspaces }
    \label{fig:ladder}
\end{figure}

\begin{figure}
    \centering
    \includegraphics[scale=0.8]{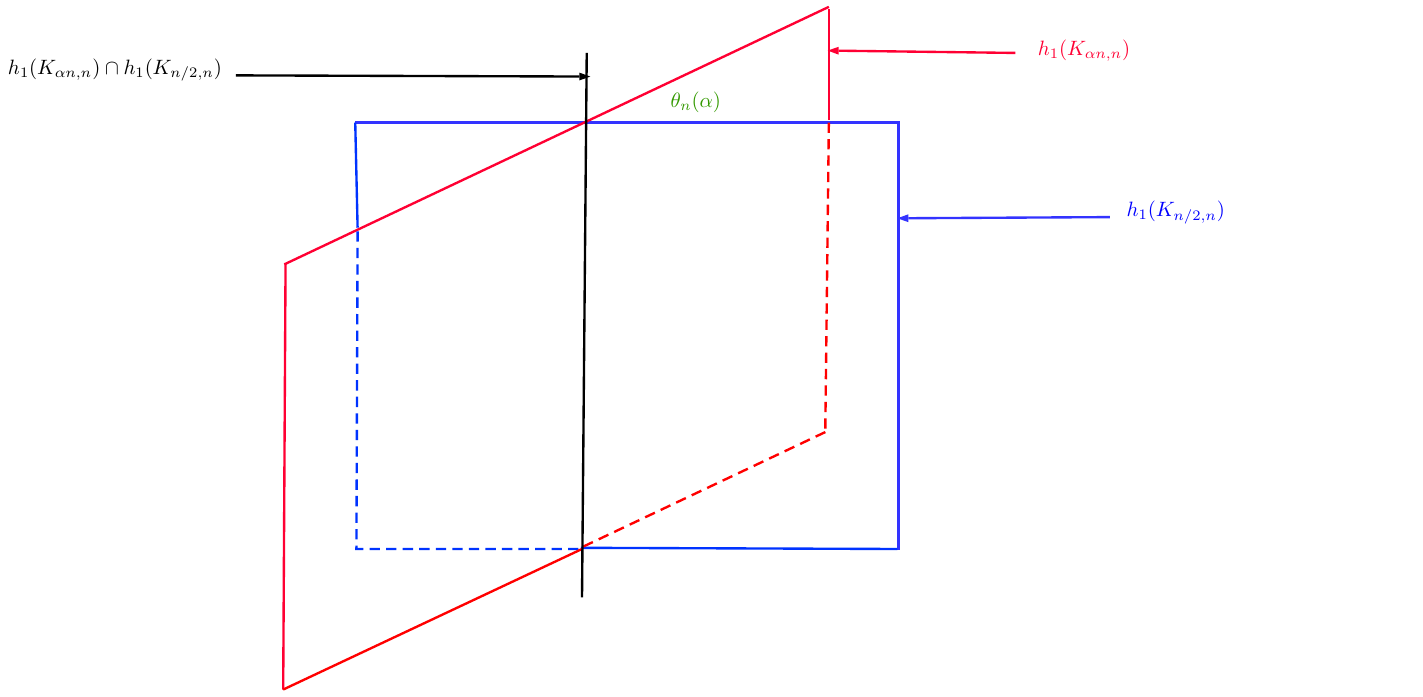}
    \caption{The subspaces $\mathcal{H}_1(K_{\alpha n,n})$ and $\mathcal{H}_1(K_{n/2,n})$.}
    \label{fig:ladder-subspaces}
\end{figure}

Let $n > 1$ be a fixed integer, and 
for $1 < m < n$ let $K_{m,n}$ denote the one dimensional simplicial complex depicted in 
Figure~\ref{fig:ladder}. Then, $\mathcal{H}_1(K_{m,n})$ is a two-dimensional subspace 
of $C_1(K)$, where $K$ denotes the union over $m, 1 < m < n$,  of all the $K_{m,n}$ as sub-complexes.

Now, the harmonic homology spaces $\mathcal{H}_1(K_{m,n})$ have a one-dimensional
subspace in common -- namely, the subspace spanned by the vector 
\[
w=[v_1,v_2] + [v_2,v_3] +\cdots +  [v_{2n-1},v_{2n}] + [v_{2n},v_1].
\]

For $0 < \alpha < \frac{1}{2}$, we consider the subspaces $\mathcal{H}_1(K_{\lfloor \alpha n\rfloor ,n})$ and 
$\mathcal{H}_1(K_{\lfloor n/2 \rfloor ,n})$. Note there is at most one non-zero principal angle 
between these subspaces since they have a one-dimensional subspace in common. Let $\theta_n(\alpha)$ denote this principal angle.
Intuitively,
one should expect that as $\alpha \rightarrow 1/2$, 
the subspaces $\mathcal{H}_1(K_{\lfloor \alpha n\rfloor ,n})$ and 
$\mathcal{H}_1(K_{\lfloor n/2 \rfloor ,n})$ should come closer to each other and
the non-zero principal angle between them should go to $0$.
Let $K_1=K_{\lfloor \alpha n\rfloor, n}$ and $K_2=K_{\lfloor n/2\rfloor, n}$. Since we are going to take the limit as $n$ goes to infinity, the floor function becomes irrelevant.

With this caveat a standard calculation which we omit shows that
\[
\cos(\theta_n(\alpha))=\dfrac{\alpha n}{\left( 2\alpha n(1-\alpha)^2 + (2n)(1-\alpha)\alpha^2 +1\right)^{\frac{1}{2}}\left(\frac{n}{2}+1 \right)^{\frac{1}{2}}
}.
\]
This implies that
$
\lim_{n \rightarrow \infty} \cos \theta_n(\alpha) = \left(\frac{\alpha}{1-\alpha}\right)^{1/2}
$,
which agrees with the intuition expressed before, since clearly
the angle $\theta_n(\alpha)$ goes to $0$, as $\alpha$ approaches $1/2$.
\end{example}

\section{Harmonic persistent homology and harmonic barcodes of a filtration}
\label{sec:harmonic-persistent-homology}
In this section we define harmonic persistent homology and the associated harmonic persistent barcodes
and prove stability results.
We first recall the notion of persistent homology and also the definition of the associated barcodes.
There is some subtlety in the definition of the latter that is explained in the following section.

\subsection{Persistent homology and barcodes}
\label{subsec:persistent}
Let $T$ be an ordered set 
(with or without  endpoints),   and
$\mathcal{F} = (K_t)_{t \in T}$,    a tuple of sub-complexes of a  finite simplicial complex $K$, such that
$s \leq t \Rightarrow K_s \subset K_t$.    We call $\mathcal{F}$ a filtration
of the  simplicial complex $K$.

\begin{convention}
\label{convention:1}
If the ordered set $T$  has end points, $a = \min(T), b = \max(T)$ (for example if $T$ is finite), then
we will formally adjoin two new elements $-\infty$ and $\infty$ to $T$ with $-\infty < c, \infty > c$,
for every $c \in T$,
and adopt the convention that $K_{-\infty} = \emptyset$ and $K_\infty = K_b$ (or equivalently
that $K_s = \emptyset$ for $s < a$ and $K_t = K_b$ for $t > b$).
\end{convention}

\begin{notation}
\label{not:inclusion}

  For $s, t \in T,   s \leq t$, and $p \geq 0$, we let $i_p^{s, t} : \HH_p (K_s) \longrightarrow
  \HH_p (K_t)$, denote the homomorphism induced by the inclusion $K_s
  \hookrightarrow K_t$.
\end{notation}

\begin{definition}[Persistent homology groups]
\label{def:persistent}
  {\cite{Edelsbrunner_survey}} For each triple $(p, s, t)  \in \Z_{\geq 0} \times T  \times T$ with 
$s \leq t$ 
  the   
  \emph{persistent homology  group},       
  $\HH_p^{s, t} (\mathcal{F})$ is defined by
  \begin{eqnarray*}
    \HH_p^{s, t} (\mathcal{F}) & = &  \mathrm{Im} (i_p^{s, t}).
  \end{eqnarray*}
  Note that $\HH_p^{s, t} (\mathcal{F}) \subset \HH_p (K_t)$,    
  and 
  $\HH_p^{s, s}(\mathcal{F}) = \HH_p (K_s)$.
\end{definition}
\begin{notation}[Persistent Betti numbers]
  We denote by $b_p^{s, t} (\mathcal{F}) = \dim_\R(\HH_p^{s, t}(\mathcal{F}))$, and call
 $b_p^{s,t}(\mathcal{F})$ $(p,s,t)$-th persistent Betti number.
 \end{notation}
 We now define barcode of a filtration.
 
\subsubsection{Barcodes of filtrations}

\begin{definition}
\label{def:birth-death}
In the lingua franca of the theory of persistent homology, for $s \leq t \in T$, and $ p \geq 0$,
	\begin{itemize}
		\item  we say that a homology class $\gamma \in \HH_p(K_{s})$ is \emph{born} at time $s$, if $\gamma \notin \HH_{p}^{s',s}(\mathcal{F})$, for any $s' < s$;
		\item for a class $\gamma \in \HH_p(K_{s})$ born at time $s$, we say that $\gamma$ \emph{dies} at time $t$, if  $i_p^{s,t'}(\gamma) \notin \HH_{p}^{s',t'}(\mathcal{F})$ for all $s', t'$ such that $s' < s < t' < t$, but $i_p^{s,t}(\gamma) \in \HH_{p}^{s'',t}(\mathcal{F})$, for some $s'' < s$.
	\end{itemize}  	
\end{definition}

First observe that it follows from Definition~\ref{def:persistent} that for all $s' \leq s \leq t$ and $p \geq 0$, 
$\HH_p^{s',t}(\mathcal{F})$ is a subspace of $\HH^{s,t}_p(\mathcal{F})$, and both are subspaces of
$\HH_p(K_t)$. This is because the homomorphism $i^{s',t}_p = i^{s,t}_p\circ i^{s',s}_p$, and so the image of 
$i^{s',t}_p$ is contained in the image of $i^{s,t}_p$.
It follows that, for $s \leq t$,  
$\sum_{s' < s} \HH^{s',t}_p(\mathcal{F}) = \bigcup_{s' < s} \HH^{s',t}_p(\mathcal{F})$ 
since the subspaces $\HH^{s',t}$ satisfy that  $\HH^{s',t} \subset \HH^{s'',t}$ for all $s' \leq s'' < s$.
In particular, setting $t=s$,  
$\sum_{s' < s} \HH^{s',s}(\mathcal{F})$ is a subspace of $\HH_p(K_s)$. \\

The following definitions are taken from \cite{Basu-Karisani} (see also the references therein and \cite[Theorem 1]{Ghrist-Henselman}).
We follow  the same notation as above and first define certain subspaces of the homology
groups $\HH_p(K_s), s \in T, p \geq 0$.

\begin{definition}[Subspaces of $\HH_p(K_s)$]
\label{def:barcode}
For $s \leq t$, and $p \geq 0$, we define
\begin{eqnarray*}
M^{s,t}_p(\mathcal{F}) &=& \sum_{s' < s} (i^{s,t}_p)^{-1}(\HH^{s',t}_p(\mathcal{F})), \\
M^{s,\infty}_p(\mathcal{F}) &=& \sum_{s \leq t' } M^{s,t'}_p(\mathcal{F}), \\
N^{s,t}_p(\mathcal{F}) &=& \sum_{s' < s\leq t' < t} (i^{s,t'}_p)^{-1}(\HH^{s',t'}_p(\mathcal{F})). \\
\end{eqnarray*}
\end{definition}

\begin{remark}
\label{rem:def:barcode:finite}
Definition~\ref{def:barcode} does not assume anything about the ordered indexing set $T$.
In the special case when $T$ is finite, for example if $T = [0,N] \subset \mathbb{Z}$, the definition of the
subspaces in Definition~\ref{def:barcode} simplifies and we have for 
$0 \leq s < t \leq N$, 
and $p \geq 0$,
\begin{eqnarray*}
M^{s,t}_p(\mathcal{F}) &=& (i^{s,t}_p)^{-1}(\HH^{s-1,t}_p(\mathcal{F})), \\
M^{s,\infty}_p(\mathcal{F}) &=& M^{s,N}_p(\mathcal{F}), \\
N^{s,t}_p(\mathcal{F}) &=&  (i^{s,t-1}_p)^{-1}(\HH^{s-1,t-1}_p(\mathcal{F}))\\
                        &=&  M_p^{s,t-1}(\mathcal{F}).
\end{eqnarray*}
\end{remark}

The ``meaning''  of the subspaces $M^{s,t}_p(\mathcal{F}), N^{s,t}_p(\mathcal{F})$ 
in terms of birth and death of homological cycles
(as per Definition~\ref{def:birth-death}) is encapsulated in the following proposition.

\begin{proposition}
	\label{prop:barcode_subspace_mean}
		\begin{remunerate}
		 
			\item
			\label{itemlabel:prop:barcode_subspace_mean:1} 
			For every  
   $s,t \in T, s < t $, 
   $M^{s,t}_p(\mathcal{F})$  is  a subspace of  $\HH_p(K_s)$ consisting of homology classes in $\HH_p(K_s)$ which are either $0$
			or 
			\begin{center}
				``are either born before time $s$, or born at time $s$ and die at time $t$ or earlier''. 
			\end{center}
			
			\item
			\label{itemlabel:prop:barcode_subspace_mean:2} 
			For every fixed $s \in T$,  $M^{s,\infty}_p(\mathcal{F})$ is a subspace of $\HH_p(K_s)$ consisting of homology classes in $\HH_p(K_s)$ which are either $0$ or
			\begin{center}
				``are either born before time $s$, or born at time $s$ and die at some time $t \geq s$''. 
			\end{center}

			\item
			\label{itemlabel:prop:barcode_subspace_mean:3} 
			Similarly,
			for every   
   $s,t  \in T, s < t$,
   $N^{s,t}_p(\mathcal{F})$  is  a subspace of $\HH_p(K_s)$ consisting of homology classes in $\HH_p(K_s)$ which are either $0$ or
			\begin{center}
				``are either born before time $s$, or born at time $s$ and die strictly earlier than $t$''.
			\end{center}
		\end{remunerate}
\end{proposition}

\begin{proof}
\begin{remunerate}
    \item 
    Let $0 \neq \gamma \in  M^{s,t}_p(\mathcal{F})$. 
    
    Since $M^{s,t}_p(\mathcal{F}) \subset \HH_p(K_s)$, and
    every element of $\HH_p(K_s)$ is born at time $s$ or earlier, it is clear that $\gamma$ is born 
    at time $s$ or earlier.
    
    If $s = \min(T)$, then using Convention~\ref{convention:1}  $\HH^{s',t}_p(\mathcal{F}) = 0$ for 
    $s' < s$, and $\gamma \in \Ker(i^{s,t}_p)$. This implies that $\gamma$ is born at time $s$ and
    dies at or before time $t$.
    
    If $s \neq \min(T)$, then
    it follows from Definition~\ref{def:barcode} that 
    \[
    \gamma \in M^{s,t}_p(\mathcal{F}) \Leftrightarrow 
    \exists s' < s \mbox{ such that } i^{s,t}_p(\gamma) \in \HH^{s',t}_p(\mathcal{F}) = \im(i^{s',t}_p).
    \]
    This implies as per Definition~\ref{def:birth-death} that
    
    Now $i^{s',t}_p = i^{s,t}_p \circ i^{s',s}_p$. So there exists $s' < s$ and $\gamma' \in \HH_p(K_{s'})$
    such that 
    \begin{equation}
    \label{eqn:prop:barcode_subspace_mean:proof:1}
    i^{s',t}_p(\gamma') = i^{s,t}_p \circ i^{s',s}_p(\gamma') = i^{s,t}_p(\gamma).
    \end{equation}

    If there exists $\gamma'$ satisfying 
    \eqref{eqn:prop:barcode_subspace_mean:proof:1}  such that
    $i^{s',s}_p(\gamma') = \gamma$
    then $\gamma$ is born before time $s$.
    
    Else, $s$ is born at time $s$.  
    Moreover, the image of $\gamma$ in $\HH^{s,t}_p(\mathcal{F})$ equals the image of
    $\gamma'$ in $\HH^{s,t}_p(\mathcal{F})$ for some $\gamma' \in \HH_p(K_{s'})$. 
    It follows that $\gamma$ is born at time $s$ and dies at time $t$ or earlier.
    
    Conversely, suppose that  $\gamma \in \HH_p(K_s)$, and
    either $\gamma$ is born before time $s$,  or $\gamma$ is born at time $s$ and dies at time $t$ or earlier.
    
    There are two possibilities. Either $\gamma$ is born before time $s$. In this case,
    there exists $s' < s$ and $\gamma' \in \HH_p(K_{s'})$ such that
    $i^{s,t}_p(\gamma) = i^{s',t}_p(\gamma')$ which implies that 
    $\gamma \in M^{s,t}_p(\mathcal{F})$. 
    
    Or, $\gamma$ is born at time $s$ and dies at time $t$ or earlier. In this case, 
    $i^{s,t}_p(\gamma) = 0$, which implies that $\gamma \in M^{s,t}_p(\mathcal{F})$.
    
    This proves Part~\eqref{itemlabel:prop:barcode_subspace_mean:1}.
    \item
    Part~\eqref{itemlabel:prop:barcode_subspace_mean:2} is follows directly from Part~\eqref{itemlabel:prop:barcode_subspace_mean:1}.
    \item
    It is clear from the definition that each $\gamma \in N^{s,t}_p(\mathcal{F})$ belongs to
    $M^{s,t'}_p(\mathcal{F})$ for some $t'$ with $s\leq t' < t$. 
    Part~\eqref{itemlabel:prop:barcode_subspace_mean:3} now follows from 
    Part~\eqref{itemlabel:prop:barcode_subspace_mean:1}.
\end{remunerate}
\end{proof}

We now define certain \emph{subquotients} of the homology groups of $\HH_p(K_s), s \in T, p \geq 0$,
in terms of the subspaces defined above in Definition~\ref{def:barcode}. 
These subquotients (which are vector spaces and therefore 
have well defined dimensions) are in bijection with equivalence classes
of homology classes ``born at time $s$ and which die at time $t$''.

\begin{definition}[Subquotients associated to a filtration]
\label{def:barcode1}
For 
$s < t$
and $p \geq 0$, we define
\begin{eqnarray*}
P^{s,t}_p(\mathcal{F}) &=& M^{s,t}_p(\mathcal{F})/N^{s,t}_p(\mathcal{F}), \\
P^{s,\infty}_p(\mathcal{F}) &=&  
\HH_p(K_s) / M^{s,\infty}_p(\mathcal{F}).
\end{eqnarray*}

We will call 
\begin{remunerate}
\item
$P^{s,t}_p(\mathcal{F})$ the \emph{space of $p$-dimensional cycles  born at time $s$ and which dies at time $t$};
and 
\item
$P^{s,\infty}_p(\mathcal{F})$ the \emph{space of $p$-dimensional cycles born at time $s$ and which never die}.
\end{remunerate}
\end{definition}

\begin{definition}[Persistent multiplicity, barcode, simple bars]
\label{def:barcode2}
We will denote for $s \in T,  t \in T \cup \{\infty \}$,
\begin{equation}
\label{eqn:def:barcode:multiplicity}
\mu^{s,t}_p(\mathcal{F}) = \dim P^{s,t}_p(\mathcal{F}),
\end{equation}
and call $\mu^{s,t}_p(\mathcal{F})$ the \emph{persistent multiplicity of $p$-dimensional cycles born at time $s$ and dying at time $t$ if $t \neq \infty$,   or never dying in case $t = \infty$}.

Finally, we will call the set
\[
\mathbf{B}_p(\mathcal{F}) = \{(s,t,\mu^{s,t}_p(\mathcal{F})) \mid \mu^{s,t}_p(\mathcal{F}) > 0\}
\]
\emph{the $p$-dimensional barcode associated to the filtration $\mathcal{F}$}. 

We will call an element
$b = (s,t,\mu^{s,t}_p(\mathcal{F})) \in \mathbf{B}_p(\mathcal{F})$ a \emph{bar of $\mathcal{F}$ of multiplicity
$\mu^{s,t}_p(\mathcal{F}$}). 
If $\mu^{s,t}_p(\mathcal{F})=1$, we will call $b$ a \emph{simple} bar.
\end{definition}

\subsection{Harmonic persistent homology and harmonic barcodes}
\label{subsec:harmonic-peristent-homology}
We now define the harmonic versions of persistent homology groups. These will all be subspaces
of $C_p(K), p \geq 0$.

Using Proposition~\ref{prop:functorial} we have for each $s,t \in T,s \leq t$, a linear map
\[
\mathfrak{i}_p^{s,t} := \mathrm{proj}_{B_p(K_t)^\perp}|_{\mathcal{H}_p(K_s)} : \mathcal{H}_p(K_s) \rightarrow \mathcal{H}_p(K_t),
\]

which makes the following diagram commute
\begin{equation}
\label{eqn:commutative}
\xymatrix{
\HH_p(K_s) \ar[r]^{i_p^{s,t}} \ar[d]^{\mathfrak{f}_p(K_s)} & \HH_p(K_t) \ar[d]^{\mathfrak{f}_p(K_t)} \\
\mathcal{H}_p(K_s) \ar[r]^{\mathfrak{i}_p^{s,t}} &\mathcal{H}_p(K_t)
}.
\end{equation}

\begin{definition}[Harmonic persistent homology subspaces]
\label{def:harmonic-persistent}
For each triple $(p, s, t)  \in \Z_{\geq 0} \times T  \times T$ with 
$s \leq t$ 
  the   
  \emph{harmonic persistent homology  subspace},       
  $\mathcal{H}_p^{s, t} (\mathcal{F})$ is defined by
  \begin{eqnarray*}
    \mathcal{H}_p^{s, t} (\mathcal{F}) & = &  \mathrm{Im} (\mathfrak{i}_p^{s, t}(\mathcal{F})) \subset C_p(K).
  \end{eqnarray*}
\end{definition}

\subsubsection{Harmonic barcodes of  filtrations}
We now give the harmonic analogs of the above spaces. They are all subspaces of $C_p(K)$ (in fact, of 
the various harmonic homology spaces $\mathcal{H}_p(K_s), s \in T$).

\begin{definition}[Harmonic barcode of a filtration]
\label{def:harmonic-barcode}
For $s \leq t$, and $p \geq 0$, we define
\begin{eqnarray*}
\mathcal{M}^{s,t}_p(\mathcal{F}) &=& \sum_{s' < s} (\mathfrak{i}^{s,t}_p)^{-1}(\mathcal{H}^{s',t}_p(\mathcal{F})), \\
\end{eqnarray*}
and for $s < t$,
\begin{eqnarray*}
\mathcal{N}^{s,t}_p(\mathcal{F}) &=& \sum_{s' < s\leq t' < t} (\mathfrak{i}^{s,t'}_p)^{-1}(\mathcal{H}^{s',t'}_p(\mathcal{F})), \\
\mathcal{P}^{s,t}_p(\mathcal{F}) &=& \mathcal{M}^{s,t}_p(\mathcal{F}) \cap \mathcal{N}^{s,t}_p(\mathcal{F})^\perp, \\
\mathcal{P}^{s,\infty}_p(\mathcal{F}) &=&  \mathcal{H}_p(K_s)  \cap   \bigcap_{s \leq t} \mathcal{M}^{s,t}_p(\mathcal{F})^\perp.
\end{eqnarray*}
\end{definition}

The vector spaces defined in Definition~\ref{def:harmonic-barcode} are isomorphic to
the corresponding ones in Definitions~\ref{def:barcode} and \ref{def:barcode1}. 
Following the same notation as in Definition~\ref{def:harmonic-barcode} we have the 
following proposition.

\begin{proposition}
\label{prop:harmonic-barcode}
\begin{eqnarray}
\label{eqn:prop:harmonic-barcode:1}
\mathcal{M}^{s,t}_p(\mathcal{F}) &\cong & {M}^{s,t}_p(\mathcal{F}), \\ 
\label{eqn:prop:harmonic-barcode:2}
\mathcal{N}^{s,t}_p(\mathcal{F}) &\cong& {N}^{s,t}_p(\mathcal{F}), \\
\label{eqn:prop:harmonic-barcode:3}
\mathcal{P}^{s,t}_p(\mathcal{F}) &\cong& {P}^{s,t}_p(\mathcal{F}), \\
\label{eqn:prop:harmonic-barcode:4}
\mathcal{P}^{s,\infty}_p(\mathcal{F}) &\cong&  {P}^{s,\infty}_p(\mathcal{F}).
\end{eqnarray}
\end{proposition}

\begin{proof}
The isomorphisms in \eqref{eqn:prop:harmonic-barcode:1} and \eqref{eqn:prop:harmonic-barcode:2} follow from the fact that the vertical arrows in the commutative
diagram \eqref{eqn:commutative} are isomorphisms (Proposition~\ref{prop:f}).

In order to prove the isomorphism in \eqref{eqn:prop:harmonic-barcode:3},  first observe that  
for $s < t$,
$\mathcal{N}^{s,t}_p(\mathcal{F})$ is a subspace of 
$\mathcal{M}^{s,t}_p(\mathcal{F})$.
The isomorphism now follows from the following simple observation.
If $W_1 \subset W_2$ are subspaces of an Euclidean  space $V$, then
$W_1^\perp \cap W_2$ equals the orthogonal complement of $W_1$ in $W_2$ 
(restricting the inner product of $V$ to $W_2$).

Finally, the isomorphism in \eqref{eqn:prop:harmonic-barcode:4} follows from the fact that
$
\sum_{s \leq t} \mathcal{M}^{s,t}_p(\mathcal{F})
$ is a subspace of 
$\mathcal{H}_p(K_s)$,
and 
\[
\left(\sum_{s \leq t} \mathcal{M}^{s,t}_p(\mathcal{F})\right)^\perp = \bigcap_{s \leq t} \mathcal{M}^{s,t}_p(\mathcal{F})^\perp,
\]
and the previous observation.
\end{proof}

In analogy with Definition~\ref{def:barcode2} we now define harmonic barcodes of filtrations. 
\begin{definition}[Harmonic barcodes]
\label{def:harmonic-barcode2}
We will call the set
\[
\mathbf{\mathcal{B}}_p(\mathcal{F}) = \{(s,t,\mathcal{P}^{s,t}_p(\mathcal{F})) \mid  \mathcal{P}^{s,t}_p(\mathcal{F})) \neq  0\}
\]
\emph{the $p$-dimensional harmonic barcode associated to the filtration $\mathcal{F}$}.
For ${b} = (s,t,\mu^{s,t}_p(\mathcal{F})) \in {B}_p(\mathcal{F})$, we will call the subspace
$\mathcal{P}^{s,t}_p(\mathcal{F}) \subset \mathcal{H}_p(K_s)$, the harmonic homology subspace associated to
$b$.
\end{definition}

\begin{remark}
Note that  
\[
\mathbf{B}_p(\mathcal{F})  \subset  T  \times \left(T \cup \{\infty\}\right)  \times  \Z_{> 0},
\]
and
\[
\mathbf{\mathcal{B}}_p(\mathcal{F})  \subset  T  \times \left(T \cup \{\infty\}\right)  \times \coprod_{0 \leq d \leq \dim C_p(K)}  \Gr(d,  C_p(K)).
\]
\end{remark}

\subsection{Harmonic persistent homology spaces of filtrations induced by functions}
\label{subsec:stability}
We now consider filtrations of finite simplicial complexes induced by functions. Our goal
is to prove that the harmonic persistent homology spaces of such filtrations are stable with respect 
to perturbations of of the functions defining them.

\subsubsection{Filtration induced by functions}

\begin{definition}[Admissible functions]
\label{def:simplex-wise}
Let $K$ be a finite simplicial complex and $f:K \rightarrow \R$. We say that $f$  
is an \emph{admissible function} if $f$ satisfies for each $\sigma,\tau \in K$,
with $\sigma \prec \tau$, $f(\sigma) < f(\tau)$.
\end{definition}

\begin{notation}[Filtration induced by an admissible function]
\label{not:induced-filtration}
If $f:K \rightarrow \R$ is an admissible function, then for each $t \in \R$,
$K_{f \leq t} = f^{-1}(-\infty,t]$ is a sub-complex of $K$,
and $(K_{f \leq t})_{t \in \R}$ is a filtration with respect to the usual order of $\R$.
We will call this filtration \emph{the filtration induced by $f$} 
and denote it by $\mathcal{F}_f$.
\end{notation}

\begin{definition}[Harmonic filtration  function]
\label{def:persfunc}
Let $K$ be a finite simplicial complex, and
let $f: K\rightarrow \R$ be 
an admissible function.

For each $p \geq 0$ we define:
\begin{align*}
\mathcal{H}_p(K,f): \R&\rightarrow \Gr(C_p(K)) = \coprod_{0 \leq d \leq \dim C_p(K)}  \Gr(d, C_p(K)) \\
t& \mapsto \h_p(K_{f \leq t}).
\end{align*}
We will denote $\mathcal{H}_p(K,f)(t)$ by $\mathcal{H}_p^t(K,f)$. 
We call $\mathcal{H}_p(K,f)$ a 
\emph{harmonic filtration function of dimension $p$}.
\end{definition}

\begin{remark}
Notice that the harmonic filtration function $\mathcal{H}_p(K,f)$ contains all the information of persistent homology. 
Once $\mathcal{H}_p(K,f)$ is known, for $s\leq t$ we have 
\[
\proj_{\mathcal{H}_p(K,f)(t)}(\mathcal{H}_p(K,f)(s))=\h_p^{s,t}(\mathcal{F}_f)\cong \HH_p^{s,t}(\mathcal{F}_f).
\]
This was already observed by Lieutier (see \cite[Lemma 2]{talk:Lieutier}).
\end{remark}

\subsubsection{Metric on the space of harmonic filtration functions}
Observe that for any finite dimensional real vector space $V$, and for $d \geq 0$, the Grassmannian
$\Gr(d,V)$ has the structure of a semi-algebraic set -- and hence, $\Gr(V) = \coprod_{0 \leq d \leq \dim V}\Gr(d,V)$
is also a semi-algebraic set. For a finite simplicial complex $K$, and $p \geq 0$, we denote by
$\mathcal{S}_p(K)$ the set of
piece-wise constant maps with at most finitely many discontinuities
$F:\R \rightarrow \Gr(C_p(K))$.

We now introduce a class of (pseudo)-metrics on $\mathcal{S}_p(K)$.
\begin{definition}
\label{def:metric}
Let $\ell >0$. For $F,G \in \mathcal{S}_p(K)$ we define
\begin{equation}
\label{eqn:def:metric}
 \dist_{K,p,\ell}(F,G) = \left(\int_\R d_{K,p}(F(t),G(t))^\ell\; dt\right)^{1/\ell}.   
\end{equation}
\end{definition}

We also define a family of semi-norms on the space of functions $f:K \rightarrow \R$.
\begin{definition}
\label{def:norm-on-functions}
Let $\ell > 0$.
Given a function $f:K \rightarrow \R$ we define 
\begin{equation}
\label{eqn:def:norm-on-functions}
||f||_\ell^{(p)}=  \left(\sum_{\sigma \in K^{\sqbracket{p}}} |f(\sigma)|^\ell\right)^{1/\ell}.   
\end{equation}
\end{definition}

We are now in a position to state our first stability result.
\begin{theorem}
\label{thm:stable}
Let $K$ be a finite simplicial complex, $f,g:K \rightarrow \R$ be admissible functions 
and $p \geq 0$.
Then,
\[
\dist_{K,p,2}(\mathcal{H}_p(K,f(t)), \mathcal{H}_p(K,g(t))) \leq    
\frac{\pi}{2} \cdot \left( ||f-g||_1^{(p)} + ||f-g||_1^{(p+1)}\right)^{1/2}.
\]
\end{theorem}

\begin{proof}
Using Corollary~\ref{cor:stability-homology} we have 
\begin{eqnarray*}
\dist_{K,p,2}(\mathcal{H}_p(K,f(t)), \mathcal{H}_p(K,g(t)))^2 &=& \int_\R d_{K,p}(\mathcal{H}_p^t(K,f),\mathcal{H}_p^t(K,g))^2 dt \\
&\leq&
\frac{\pi^2}{4}\cdot\int_\R \left( \sum_{\sigma \in K^{\sqbracket{p}} \cup K^{\sqbracket{p+1}}} |\chi_{K_{f \leq t}}(\sigma) - \chi_{K_{g \leq t}}(\sigma)|        dt \right)\\
&=& 
\frac{\pi^2}{4}\cdot\left( \sum_{\sigma \in K^{\sqbracket{p}} \cup K^{\sqbracket{p+1}}} \int_\R  |\chi_{K_{f \leq t}}(\sigma) - \chi_{K_{g \leq t}}(\sigma)|        dt \right)\\
&=&
\frac{\pi^2}{4}\cdot\left( \sum_{\sigma \in K^{\sqbracket{p}}} |f(\sigma) - g(\sigma)|   +     
\sum_{\sigma \in K^{\sqbracket{p}+1}} |f(\sigma) - g(\sigma)|  \right)\\
&=&
\frac{\pi^2}{4}\cdot\left( ||f - g||_1^{(p)} + ||f - g ||_1^{(p+1)}  \right).
\end{eqnarray*}
The theorem follows.
\end{proof}

We prove now a similar stability theorem for the harmonic persistent homology spaces
\[
\mathcal{H}_p^{s,t}(K,f) := \mathcal{H}_p^{s,t}(\mathcal{F}_f), 
\mathcal{H}_p^{s,t}(K,g) := \mathcal{H}_p^{s,t}(\mathcal{F}_g).
\]

For convenience  (and in order to make sure that the distance does not blow up
for trivial reasons) we will only consider admissible 
functions $f: K \rightarrow \R$ taking values in $[0,1]$. This is not a serious restriction, since $K$ is a finite complex and one can always scale and translate functions without affecting the induced finite filtration of $K$ other than the index. 

We define:
\begin{definition}
\label{def:persistent-harmonicdist}
Let $\ell > 0$.
For $F,G \in \mathcal{S}_p(K)$ we define
\begin{equation}
\label{eqn:persistent-harmonicdist}
 \dist_{\h,p,\ell}^{\persistent}(F,G) = \left(\int_{0}^{1} \int_{0}^{t}  
 d_{K,p}(F^{s,t}, G^{s,t})^\ell ds dt \right)^{1/\ell},
\end{equation}
where 
\[
F^{s,t} = \proj_{F(t)}F(s), G^{s,t} = \proj_{G(t)}G(s).
\]
\end{definition}

\begin{theorem}
\label{thm:stable-persistent}
Let $K,f,g$ be as above and $p \geq 0$.
Then,
\[
\dist_{\h,p,1}^{\persistent}(\mathcal{H}_p(K,f), \mathcal{H}_p(K,g)) \leq    \pi \cdot \left( ||f-g||_1^{(p)} + ||f-g||_1^{(p+1)}\right).
\]
\end{theorem}

Before proving Theorem~\ref{thm:stable-persistent} we need a lemma.

\begin{lemma}
\label{lem:stable-persistent}
Let $V$ be a finite dimensional inner product space and $P,V_1,V_2$, subspaces of $V$.
Let for $i =1,2$, $Z_i = \proj_{V_i}(P)$ and $\Delta_i = \dim(V_i) - \dim (V_1 \cap V_2)$.
Then for $i=1,2$,
\[
\dim Z_i - \dim(Z_1 \cap Z_2) \leq \Delta_1 + \Delta_2.
\]
\end{lemma}

\begin{proof}
Let $V' = V_1+V_2$,
and let 
\[
\pi_1 = \proj_{V_1}|_P,
\pi_2 = \proj_{V_2}|_P,
\pi = \proj_{V_1 \cap V_2}|_P:P  \rightarrow V',
\]
denote the orthogonal projections from $P$ on to $V_1,V_2$, and $V_1 \cap V_2$ respectively.
Denote for $i=1,2$, $r_i = \rank(\pi_i)$, and  $r = \rank(\pi)$.

Using the notation introduced above we have:
\begin{eqnarray}
\nonumber
\dim(Z_1\cap Z_2) &=& \dim(Z_1) + \dim(Z_2) - \dim(Z_1 + Z_2) \\
\label{eqn:lem:stable-persistent:1}
&=& r_1 + r_2 - \dim(Z_1 + Z_2).
\end{eqnarray}

Now let for $i=1,2$, $W_i$ be the orthogonal complement of $V_1 \cap V_2$ in $V_i$.
Thus, for $i=1,2$ we have an orthogonal decomposition
$V_i = (V_1 \cap V_2) \oplus W_i$,
so that $\Delta_i = \dim(W_i)$,
and a decomposition
$\pi_i= \pi \oplus \pi_i'$, where $\pi_i' = \proj_{W_i}|_P$.

Thus,
\begin{eqnarray*}
Z_1 + Z_2 &=& \image(\pi_1) + \image(\pi_2) \\
&=& \image(\pi) + \image(\pi_1') + \image(\pi_2').
\end{eqnarray*}

Hence,
\begin{eqnarray}
\nonumber
\dim(Z_1 + Z_2) &\leq& \rank(\pi) + \rank(\pi_1') + \rank(\pi_2') \\
\label{eqn:lem:stable-persistent:2}
&\leq& r + \Delta_1 + \Delta_2. 
\end{eqnarray}

Using inequalities \eqref{eqn:lem:stable-persistent:1}, \eqref{eqn:lem:stable-persistent:2}, and the 
fact that for $i,1,2$, $r \leq r_i$ (since $\pi$ factors through each of the $\pi_i$'s), we obtain
\begin{eqnarray*}
\dim(Z_1\cap Z_2) &\geq& r_1 + r_2 - r - \Delta_1 - \Delta_2 \\
&\geq& r_i - \Delta_1 - \Delta_2 \\
&=& \dim Z_i - \Delta_1 - \Delta_2,
\end{eqnarray*}
from which the lemma follows.
\end{proof}

\begin{proof}[Proof of Theorem~\ref{thm:stable-persistent}]
First observe that  the subspace $\mathcal{H}_p^{s,t}(K,f)$ (resp. $\mathcal{H}_p^{s,t}(K,g)$) is equal to the orthogonal projection
of $\mathcal{H}_p^{s}(K,f)$ into $\mathcal{H}_p^{t}(K,f)$ (respectively, $\mathcal{H}_p^{t}(K,g)$).

Denote 
\begin{equation}
\label{eqn:Fp-Gp}
F_p = \mathcal{H}_p(K,f), G_p = \mathcal{H}_p(K,g).
\end{equation}

Let 
\begin{align*}
U_f &= \mathcal{H}_p^{s}(K,f), & U_g &= \mathcal{H}_p^{s}(K,g),\\
V_f &= \mathcal{H}_p^{t}(K,f), & V_g &= \mathcal{H}_p^{t}(K,g),\\
W_f &= \proj_{V_f}(U_f), & W_g &= \proj_{V_g}(U_g), \\
Z_f &= \proj_{V_f}(U_f \cap U_g), &Z_g &= \proj_{V_g}(U_f \cap U_g).
\end{align*}

Using triangle inequality we get
\begin{eqnarray*}
d_{K,p}(F_p^{s,t}, G_p^{s,t}) &=&  d_{K,p}(W_f,W_g) \\
&\leq& d_{K,p}(W_f,Z_f) + d_{K,p}(Z_f,Z_g) + d_{K,p}(W_g,Z_g)
\end{eqnarray*}
(see Eqn.\  \eqref{eqn:Fp-Gp} and Definition~\ref{def:persistent-harmonicdist} for the definition of $F^{s,t}_p, G^{s,t}_p$). 
It is easy to verify from the definition of the metric $d_{K,p}$ that orthogonal projection does not increase
the distance between subspaces if one is a subspace of the other.
Hence,
\begin{eqnarray*}
d_{K,p}(W_f,Z_f) &\leq & d_{K,p}(U_f,U_f \cap U_g) \\
&=& \frac{\pi}{2} \cdot (\dim U_f - \dim (U_f \cap U_g))^{1/2}, \\
d_{K,p}(W_g,Z_g) &\leq & d_{K,p}(U_g,U_f \cap U_g) \\
&=& \frac{\pi}{2} \cdot (\dim U_g - \dim (U_f \cap U_g))^{1/2}.
\end{eqnarray*}

Moreover, using Lemma~\ref{lem:stable-persistent},
(with $P = U_f \cap U_g$)
we have that

\begin{eqnarray*}
d_{K,p}(Z_f,Z_g) &\leq & \frac{\pi}{2}\left(\max(\dim Z_f, \dim Z_g) - \dim (Z_f \cap Z_g)\right)^{1/2} \\
&\leq  & \frac{\pi}{2}\left((\dim V_f- \dim (V_f \cap V_g)) + (\dim V_g- \dim (V_f \cap V_g)) \right)^{1/2}.
\end{eqnarray*}

Now it follows from Lemma~\ref{lem:stability-homology} that,

\begin{eqnarray*}
\dim U_f - \dim (U_f \cap U_g) &=&
\dim \mathcal{H}_p^s(K,f) - \dim (\mathcal{H}_p^s(K,f) \cap \mathcal{H}_p^s(K,g))\\
&\leq&
\card\left(K_{f \leq s}^{\sqbracket{p}} \setminus  K_{g \leq s}^{\sqbracket{p}}\right) + \card\left(K_{g \leq s}^{\sqbracket{p+1}} \setminus K_{f \leq s}^{\sqbracket{p+1}}\right),\\
\dim U_g - \dim (U_f \cap U_g) &=& 
\dim \mathcal{H}_p^s(K,g) - \dim (\mathcal{H}_p^s(K,f) \cap \mathcal{H}_p^s(K,g)) \\
&\leq&
\card\left(K_{g \leq s}^{\sqbracket{p}} \setminus  K_{f \leq s}^{\sqbracket{p}}\right) + \card\left(K_{f \leq s}^{\sqbracket{p+1}} \setminus K_{g \leq s}^{\sqbracket{p+1}}\right), \\
\dim V_f - \dim (V_f \cap V_g)  &=& \dim \mathcal{H}_p^t(K,f)- \dim (\mathcal{H}_p^t(K,f) \cap \mathcal{H}_p^t(K,g)) \\
&\leq&
\card\left(K_{f \leq t}^{\sqbracket{p}} \setminus  K_{g \leq t}^{\sqbracket{p}}\right) + \card\left(K_{g \leq t}^{\sqbracket{p+1}} \setminus  K_{f \leq t}^{\sqbracket{p+1}}\right), \\
\dim V_g - \dim (V_f \cap V_g) &=&
\dim \mathcal{H}_p^t(K,g)- \dim (\mathcal{H}_p^t(K,f) \cap \mathcal{H}_p^t(K,g)) \\
&\leq&
\card\left(K_{g \leq t}^{\sqbracket{p}} \setminus  K_{f \leq t}^{\sqbracket{p}}\right) + \card\left(K_{f \leq t}^{\sqbracket{p+1}} \setminus K_{g \leq t}^{\sqbracket{p+1}}\right).
\end{eqnarray*}

It follows from the above inequalities that
\begin{eqnarray}
\nonumber
d_{K,p}(F_p^{s,t}, G_p^{s,t}) &\leq&
\frac{\pi}{2} \cdot \left(\Delta_{f,g}^{(p)}(s)^{1/2} + \Delta_{g,f}^{(p)}(s)^{1/2} + (\Delta_{f,g}^{(p)}(t) + \Delta_{g,f}^{(p)}(t))^{1/2}\right)\\
\label{eqn:thm:stable-persistent:0}
&\leq&
\frac{\pi}{2} \cdot \left(\Delta_{f,g}^{(p)}(s) + \Delta_{g,f}^{(p)}(s) + \Delta_{f,g}^{(p)}(t) + \Delta_{g,f}^{(p)}(t)\right),
 \end{eqnarray}
where 
for any two functions $h_1,h_2:K \rightarrow \R$ and $r \in \R$,
\[
\Delta_{h_1,h_2}^{(p)}(r) = 
\card\left(K_{h_1 \leq r}^{\sqbracket{p}} \setminus  K_{h_2 \leq r}^{\sqbracket{p}}\right) + \card\left(K_{h_2 \leq r}^{\sqbracket{p+1}} \setminus K_{h_1 \leq r}^{\sqbracket{p+1}}\right).
\]

Using Definition~\ref{def:persistent-harmonicdist} we have
\begin{eqnarray}
\label{eqn:thm:stable-persistent}
\dist_{\h,p,1}^{\persistent}(\mathcal{H}_p(K,f), \mathcal{H}_p(K,g)) 
&=&
\int_0^1 \int_0^t d_{K,p}(\mathcal{H}_p^{s,t}(K,f), \mathcal{H}_p^{s,t}(K,g)) ds dt.
\end{eqnarray}

Using equality \eqref{eqn:thm:stable-persistent} and inequality \eqref{eqn:thm:stable-persistent:0} we get
\begin{eqnarray*} 
\dist_{\h,p,1}^{\persistent}(\mathcal{H}_p(K,f), \mathcal{H}_p(K,g)) &\leq &
\frac{\pi}{2} \cdot \int_0^1 \int_0^t \left(\Delta_{f,g}^{(p)}(s) + 
 \Delta_{g,f}^{(p)}(s) + \Delta_{f,g}^{(p)}(t) + \Delta_{g,f}^{(p)}(t) \right) ds dt \\
&\leq &
\frac{\pi}{2} \cdot \left(
\int_0^1 \left(\Delta_{f,g}^{(p)}(t) + \Delta_{g,f}^{(p)}(t)\right) dt + \int_0^1 \left(\Delta_{f,g}^{(p)}(s) + \Delta_{g,f}^{(p)}(s)\right)  ds\right) \\
&=&
\pi \cdot \int_0^1 \left(\Delta_{f,g}^{(p)}(s) + \Delta_{g,f}^{(p)}(s)\right) ds \\
&=&
\pi \cdot \left( ||f-g||_1^{(p)} + ||f-g||_1^{(p+1)}\right).
\end{eqnarray*}
This completes the proof.
\end{proof}

\subsection{Stability of harmonic persistent barcodes}
We have proved the stability of the harmonic homology, as well as the persistent homology
subspaces of filtrations of finite simplicial complexes induced by admissible
functions (Theorems~\ref{thm:stable},
and \ref{eqn:thm:stable-persistent}). 
We now study the stability of the harmonic barcodes. 

Let $K$ be a finite simplicial complex and  
let $\mathcal{F}$ denote a finite filtration
$K_0 \subset \cdots \subset K_N= K$.
By convention we will assume that $K_s = \emptyset$ for $s <0$ and 
$K_t = K$ for $t \geq N$.

The harmonic barcode subspaces $\mathcal{P}^{s,t}_p(\mathcal{F})$  (see Definition~\ref{def:harmonic-barcode2}) 
are subspaces of $\mathcal{H}_p(K_s)$ (corresponding to the birth of 
the homology classes corresponding to the bar
${b} = (s,t,\mu^{s,t}_p(\mathcal{F}))$ (assuming
$\mu^{s,t}_p(\mathcal{F}) \neq 0$). 

\begin{definition}
\label{def:generic}
 We will say that a bar ${b} = (s,t,\mu^{s,t}_p(\mathcal{F})) \in \mathbf{B}_p(\mathcal{F})$ is \emph{generic} if it satisfies the following two conditions.
\begin{enumerate}
    \item  \label{itemlabel:generic:1}
    $b$ is simple, i.e. $\mu^{s,t}_p(\mathcal{F}) =1$;
    \item 
    \label{itemlabel:generic:2}
    for every $t' > s, t' \neq t$, $\mu^{s,t'}_p(\mathcal{F}) = 0$ (so no other bar in $\mathbf{B}_p(\mathcal{F})$ has birth time $s$)
    \footnote{Property \ref{itemlabel:generic:2} in Definition~\ref{def:generic}  was missing from the previous version of the paper and was added after a counter-example was pointed out to us by  Aziz Burak Guelen \cite{Guelen} implying that Property ~\ref{itemlabel:generic:1} was not enough to ensure the validity of Proposition~\ref{prop:tilde-P}.}.
\end{enumerate}
\end{definition}

We show below that in the case the bar $b$ is generic
it also makes sense to associate a subspace of $\mathcal{H}_p(K_{t-1})$
representing the bar $b$ just before its  death.
With the notation introduced above:

\begin{proposition}
\label{prop:tilde-P}
Let $0 \leq s < t \leq N$, 
and suppose that the bar $b = (s,t,1) \in \mathbf{B}_p(\mathcal{F})$ is generic. 
Then,
the map $i_p^{s,t-1}$ induces an isomorphism
\[
P_p^{s,t}(\mathcal{F}) \rightarrow \HH^{s,t-1}_p(\mathcal{F})/\HH^{s-1,t-1}_p(\mathcal{F}).
\]
\end{proposition}

\begin{proof}
Recall that  (see Remark~\ref{rem:def:barcode:finite}) for $0 \leq s < t \leq N$,
\begin{eqnarray*}
M^{s,t}_p(\mathcal{F}) &=& ({i}_p^{s,t})^{-1}(\HH_p^{s-1,t}(\mathcal{F})\\
N^{s,t}_p(\mathcal{F})) &=& (i_p^{s,t-1})^{-1}(\HH_p^{s-1,t-1}(\mathcal{F})), \\
&=&  M^{s,t-1}_p(\mathcal{F}),   \mbox{ and } \\
P^{s,t}_p(\mathcal{F})) &=& M^{s,t}_p(\mathcal{F}) /N^{s,t}_p(\mathcal{F}) \\
&=& M^{s,t}_p(\mathcal{F})/M^{s,t-1}_p(\mathcal{F}).
\end{eqnarray*}

Notice that we have a filtration of subspaces, 
\[
\HH_p(K_s) \supset M^{s,N}_p(\mathcal{F}) \supset \cdots  \supset M^{s,t}_p(\mathcal{F})  \supset M^{s,t-1}_p(\mathcal{F}) \supset \cdots \supset 
M^{s,s+1}_p(\mathcal{F}) \supset M^{s,s}_p(\mathcal{F}) = \HH_p^{s-1,s}(\mathcal{F}).
\]

It follows from Property~\ref{itemlabel:generic:2} in Definition~\ref{def:generic} and Definition~\ref{def:barcode1}, and the fact that $b$ is assumed to be
generic, that all the containments in the above sequence other than $M^{s,t}_p(\mathcal{F})  \supset M^{s,t-1}_p(\mathcal{F})$
are equalities.  Moreover,  it follows from Property~\ref{itemlabel:generic:1} in Definition~\ref{def:generic} that
the index of $M^{s,t-1}_p(\mathcal{F})$ in $M^{s,t}_p(\mathcal{F})$ equals $1$.
Thus, $P_p^{s,t}(\mathcal{F}) \cong \HH_p(K_s)/\HH_p^{s-1,s}(\mathcal{F})$, and 
it is clear that $i_p^{s,t-1}$ induces a surjective map $P_p^{s,t}(\mathcal{F}) \rightarrow \HH^{s,t-1}_p(\mathcal{F})/\HH^{s-1,t-1}_p(\mathcal{F})$. It thus suffices to prove that
$\dim \HH_p^{s,t-1}(\mathcal{F})- 
\dim \HH_p^{s-1,t-1}(\mathcal{F})
= 1
$.

\hide{
Since, 
\begin{eqnarray*}
P^{s,t}_p(\mathcal{F})) &=& M^{s,t}_p(\mathcal{F}) /M^{s,t-1}_p(\mathcal{F}) \\
&\cong& 
({i}_p^{s,t-1})^{-1}(\HH_p^{s,t-1}(\mathcal{F}))
/({i}_p^{s,t-1})^{-1}(\HH_p^{s-1,t-1}(\mathcal{F})),
\end{eqnarray*}
}

It is immediate that
\begin{eqnarray*}
\dim P^{s,t}_p(\mathcal{F}))
&\geq& 
\dim \HH_p^{s,t-1}(\mathcal{F}) /\HH_p^{s-1,t-1}(\mathcal{F})\\
&=& 
\dim \HH_p^{s,t-1}(\mathcal{F})- 
\dim \HH_p^{s-1,t-1}(\mathcal{F}).
\end{eqnarray*}
Since $\dim {P}_p^{s,t}(\mathcal{F})) = 1$,
\[
\dim \HH_p^{s,t-1}(\mathcal{F})- 
\dim \HH_p^{s-1,t-1}(\mathcal{F})
\leq 1.
\]

If
\[
\dim \HH_p^{s,t-1}(\mathcal{F})- 
\dim \HH_p^{s-1,t-1}(\mathcal{F})
= 0,
\]
then $\HH_p^{s,t-1}(\mathcal{F}) = \HH_p^{s-1,t-1}(\mathcal{F})$,
which would imply that $M^{s,t}_p(\mathcal{F}) = M^{s,t-1}_p(\mathcal{F})$ and hence
\[\dim {P}_p^{s,t}(\mathcal{F})) = 0.\]

So,
\[
\dim \HH_p^{s,t-1}(\mathcal{F})- 
\dim \HH_p^{s-1,t-1}(\mathcal{F})
= 1.
\]
\end{proof}

Proposition~\ref{prop:tilde-P} motivates the following definition.
\begin{definition}[Terminal harmonic homology subspace associated to a 
generic 
bar]
\label{def:terminal}
Let
$
b = (s,t,1) \in \mathbf{B}_p(\mathcal{F})
$
be a 
generic
bar.

We denote 
\[
\widehat{\mathcal{P}}^{s,t}_p(\mathcal{F}) = 
\mathcal{H}^{s,t-1}_p(\mathcal{F}) \cap \mathcal{H}^{s-1,t-1}_p(\mathcal{F})^\perp.
\]

We call $\widehat{\mathcal{P}}^{s,t}_p(\mathcal{F})$ the \emph{terminal harmonic homology  subspace associated to ${b}$}. We will refer to $\mathcal{P}^{s,t}_p(\mathcal{F})$  itself as the \emph{initial} harmonic homology  subspace associated to ${b}$.

\end{definition}

For technical reasons we will prove stability of the terminal rather than the (initial)
harmonic subspaces  associated to generic bars.

We first define an appropriate notion of distance between the harmonic barcodes of two
different filtrations using the terminal harmonic subspaces.
The distance measured introduced for proving stability of the harmonic
homology subspaces and also the harmonic persistent homology subspaces 
(Theorems~\ref{thm:stable} and \ref{thm:stable-persistent})  
are in the form of an integral
(see Definitions~\ref{def:metric} and \ref{def:persistent-harmonicdist}).
Since for a filtration
$\mathcal{F}$ of a finite simplicial complex $K$, the subspaces
$\mathcal{P}^{s,t}_p(\mathcal{F})$ will be non-zero only for a finitely
many pairs $(s,t)$, the integral form of the distance function is not suitable.

For two finite filtrations  $\mathcal{F},\mathcal{G}$, indexed by $[N]$,
we will use the averaged sum
(see Theorem~\ref{thm:stable-persistent-terminal} below)
$
\frac{1}{\binom{N+1}{2}}\cdot \sum_{0 \leq i < j \leq N }  
d_{C_p(K)}(\widehat{\mathcal{P}}_p^{i,j}(\mathcal{F}), 
\widehat{\mathcal{P}}_p^{i,j}(\mathcal{G}))^2
$
as a measure of distance between the harmonic barcodes of $\mathcal{F},\mathcal{G}$ in dimension $p$.

We prove the following theorem.
\begin{theorem}[Stability of harmonic barcodes]
\label{thm:stable-persistent-terminal}
Let $\mathcal{F},\mathcal{G}$ 
denote filtrations
$K_0 \subset \cdots \subset K_N = K$,
$K_0' \subset \cdots \subset K_N' = K$
of a finite simplicial
complex $K$.

Let
$f,g:K \rightarrow [0,1]$ be the  admissible maps defined by
\begin{eqnarray*}
f(\sigma) &=& \frac{1}{N} \cdot \min \{s \in [N] \;\mid\; \sigma \in K_s\}, \\
g(\sigma) &=& \frac{1}{N} \cdot \min \{s \in [N] \;\mid\; \sigma \in K_s'\},
\end{eqnarray*}
each $\sigma \in K$.

Moreover, suppose that 
all bars in $\mathbf{B}_p(\mathcal{F})$ and $\mathbf{B}_p(\mathcal{G})$ are generic.

Then,
\begin{eqnarray*}
\frac{1}{\binom{N+1}{2}}\cdot \sum_{0 \leq i < j \leq N}  
 d_{C_p(K)}(\widehat{\mathcal{P}}_p^{i,j}(\mathcal{F}), 
\widehat{\mathcal{P}}_p^{i,j}(\mathcal{G}))^2
&\leq &  2 \cdot 
\frac{\pi^2}{4} \cdot
\dist_{\h,p,1}^{\persistent}(\mathcal{H}_p(K,f), \mathcal{H}_p(K,g)) \\ 
&\leq &  \frac{\pi^3}{2} \left( ||f-g||_1^{(p)} + ||f-g||_1^{(p+1)}\right).
\end{eqnarray*}
\end{theorem}

Before proving Theorem~\ref{thm:stable-persistent-terminal} we first prove a proposition
that is key to the proof of the theorem. This proposition might be of independent interest since it shows that the distance (angle)  between terminal harmonic subspaces associated to generic bars is bounded from above by the distances between
certain harmonic persistent homology subspaces of the two filtrations.

\begin{proposition}
\label{prop:principal}
Let $\mathcal{F},\mathcal{G}$ 
denote filtrations
$K_0 \subset \cdots \subset K_N = K$,
$K_0' \subset \cdots \subset K_N' = K$
of a finite simplicial
complex $K$, and let $p \geq 0$.

Suppose that $b=(i,j,1)$ belongs to both $\mathbf{B}_p(\mathcal{F})$ and $\mathbf{B}_p(\mathcal{G})$
and is generic in both.
Then,
\begin{multline*}
d_{C_p(K)}(\widehat{\mathcal{P}}_p^{i,j}(\mathcal{F})),
\widehat{\mathcal{P}}_p^{i,j}(\mathcal{G})))^2 \leq 
\frac{\pi^2}{4}\cdot (d_{C_p(K)}(\mathcal{H}^{i,j-1}_p(\mathcal{F}),
\mathcal{H}^{i,j-1}_p(\mathcal{G})) \\
+ d_{C_p(K)}(\mathcal{H}^{i-1,j-1}_p(\mathcal{F}),\mathcal{H}^{i-1,j-1}_p(\mathcal{G}))
).
\end{multline*}
\end{proposition}

Before proving Proposition~\ref{prop:principal} we need a lemma.

\begin{lemma}
\label{lem:principal}
Let $W,W' \subset V$ be two subspaces of an Euclidean space $V$, 
with $\dim W = \dim W'$. Suppose $U \subset W, U' \subset W'$ be codimension one
subspaces. Let $P = W \cap U^\perp, P' = W' \cap U'^\perp$,
$\alpha$ the principle angle between $P,P'$, 
$\theta_0$ denote the largest principle angle between $W,W'$, and 
$\theta_1,\ldots,\theta_N$ denote the principle angles between $U,U'$.
Then,
$
\alpha^2 \leq \frac{\pi^2}{4} \cdot \left(\sum_{i=0}^{N} \theta_i^2\right).
$
\end{lemma}
\begin{proof}
Let $e_1,\ldots,e_N$ (resp. $e_1',\ldots,e_N'$) be an orthonormal basis of $U$ (resp. $U'$) such the angle between $e_i,e_i'$ equals $\theta_i$. 
Let $e_0,e_0'$ be a unit
vectors spanning $P_1,P_1'$ respectively, such that 
$\la e_0,e_0'\ra =  \cos \alpha$. 

Now,
$
e_0' = (\cos \alpha) e_0 + \sum_{i=1}^N a_i e_i + e
$,
where 
$
a_i = \la e_0',e_i\ra, 1 \leq i \leq N
$ and
$
e = \proj_{W^\perp} e_0'
$.

Let $f$ be a unit length vector in $W$ such that the angle $\theta$ between
the $P' = \spanof(e_0')$ and $\spanof(f)$ is the smallest possible. Then, 
using Lemma~\ref{lem:singular},
$\theta \leq \theta_0$, since $\theta_0$ is the largest principle angle between
$W',W$. 

Since $e$ is orthogonal to $f$, it follows that the angle
between $\spanof(e)$ and $P' = \spanof(e_0')$ is bounded between
$\pi/2 - \theta$ and $\pi/2 + \theta$, and using the inequality
from the previous paragraph, we obtain that this angle is 
between $\pi/2 - \theta_0$ and $\pi/2 + \theta_0$.
It follows that
$
||e||^2  \leq  \sin^2 \theta_0
$.
Now,
$
\la e_0', e_i'\ra = 0, \la e_i',e_i \ra = \cos \theta_i,
$
which implies that the angle between $e_0',e_i$ is between $\pi/2 - \theta_i, \pi/2 + \theta_i$. 
Hence, 
$
|a_i| \leq \sin \theta_i, 1 \leq i \leq N
$
implying
$
1 = \cos^2\alpha + \sum_{i=0}^N a_i^2 \leq \cos^2\alpha + \sum_{i=0}^N \sin^2 \theta_i.
$
It follows that
$
\sin^2 \alpha \leq  \sum_{i=0}^N \sin^2 \theta_i
$.
Since,
$
\frac{2}{\pi} x \leq \sin x \leq x, 
$
for $0 \leq x \leq \frac{\pi}{2}$, 
we get that 
$
\frac{4}{\pi^2} \alpha^2 \leq \sum_{i=0}^N \theta_i^2
$
from which the stated inequality follows immediately.
\end{proof}

\begin{proof}[Proof of Proposition~\ref{prop:principal}]
Since 
$
d_{C_p(K)}(\widehat{\mathcal{P}}_p^{i,j}(\mathcal{F})),
\widehat{\mathcal{P}}_p^{i,j}(\mathcal{G})))^2
$ 
is at most $\frac{\pi^2}{4}$ 
we can assume that
$
\dim \mathcal{H}^{i,j-1}_p(\mathcal{F}) = \dim \mathcal{H}^{i,j-1}_p(\mathcal{G})
$
and 
$
\dim \mathcal{H}^{i-1,j-1}_p(\mathcal{F}) = \dim \mathcal{H}^{i-1,j-1}_p(\mathcal{G}).
$
Otherwise, the claimed inequality is true.

Now apply Lemma~\ref{lem:principal} with
\[
W = \mathcal{H}^{i,j-1}_p(\mathcal{F}), 
W' = \mathcal{H}^{i,j-1}_p(\mathcal{G}), 
U =  \mathcal{H}^{i-1,j-1}_p(\mathcal{F}), 
U' =  \mathcal{H}^{i-1,j-1}_p(\mathcal{G})
.
\]
\end{proof}

\begin{proof}[Proof of Theorem~\ref{thm:stable-persistent-terminal}]
Follows from Proposition~\ref{prop:principal},
Definition~\ref{def:persistent-harmonicdist},  and 
Theorem~\ref{thm:stable-persistent}.
\end{proof}

\section{Essential simplices and harmonic representative}
\label{sec:essential}
We now make precise the notion of essential simplices referred to previously.

Let $K$ be a finite simplicial complex and 
$\mathcal{F} = (K_t)_{t \in T}$
denote a filtration
of $K$.

Let $\phi_p^s = \phi_p(K_s): Z_p(K_s) \rightarrow \HH_p(K_s)$ be the canonical surjection. 
For every $s \in T$ we denote,
\begin{eqnarray*}
\widetilde{M}_p^{s,\infty}(\mathcal{F}) &=& \bigcup_{s \leq t} (\phi_p^s)^{-1}(M_p^{s,t}(\mathcal{F})),
\end{eqnarray*}
and for  
$s < t$, 
we denote 
\begin{eqnarray*}
\widetilde{M}_p^{s,t}(\mathcal{F}) &=& (\phi_p^s)^{-1}(M_p^{s,t}(\mathcal{F})),\\
\widetilde{N}_p^{s,t}(\mathcal{F}) &=& (\phi_p^s)^{-1}(N_p^{s,t}(\mathcal{F})).
\end{eqnarray*}

\begin{proposition}
\label{prop:harmonic-rep}
Let $K$ be a finite simplicial complex and $\FF=(K_t)_{t \in T}$a filtration.
Let $s \in T, t \in T \cup \{\infty\}$, 
$s < t$.
\begin{remunerate}
    \item If $t=\infty$ the following diagram is commutative and all maps are isomorphisms.
    \[
    \xymatrix{
    Z_p(K_s)/\widetilde{M}_p^{s,\infty}(\mathcal{F}) \ar[rr]^{\phi_p^s}\ar[rd]^{f_p}&& P^{s,\infty}_p(\mathcal{F}) \ar[ld]^{g_p}\\
    &\mathcal{P}^{s,\infty}_p(\mathcal{F})&
    },
    \]
    where 
    \begin{eqnarray*}
    \phi_p^s(z + \widetilde{M}_p^{s,\infty}(\mathcal{F})) &=& [z] +  M_p^{s,\infty}(\mathcal{F}), \\
    f_p(z + \widetilde{M}_p^{s,\infty}(\mathcal{F})) &=& \proj_{\widetilde{M}_p^{s,t}(\mathcal{F})^\perp}(z), \\
    g_p([z] +  M_p^{s,\infty}(\mathcal{F})) &=& \proj_{\widetilde{M}_p^{s,\infty}(\mathcal{F})^\perp}(z).
    \end{eqnarray*}
    (here we denote for a cycle $z \in Z_p(K_s)$, by $[z] = z +B_p(K_s)$ the homology class of $z$).

    \item If $t \neq \infty$ the following diagram is commutative and all maps are isomorphisms.
    \[
    \xymatrix{
   \widetilde{M}_p^{s,t}(\mathcal{F}) /\widetilde{N}_p^{s,t}(\mathcal{F}) \ar[rr]^{\phi_p}\ar[rd]^{f_p}&& P^{s,t}_p(\mathcal{F}) \ar[ld]^{g_p}\\
    &\mathcal{P}^{s,t}_p(\mathcal{F})&
    },
    \]
    where 
    \begin{eqnarray*}
    \phi_p(z + \widetilde{N}_p^{s,t}(\mathcal{F})) &=& [z] + N_p^{s,t}(\mathcal{F}), \\
    f_p(z + \widetilde{N}_p^{s,t}(\mathcal{F})) &=& \proj_{\widetilde{N}_p^{s,t}(\mathcal{F})^\perp}(z), \\
    g_p([z] + N_p^{s,t}(\mathcal{F})) &=& \proj_{\widetilde{N}_p^{s,t}(\mathcal{F})^\perp}(z).
    \end{eqnarray*}
\end{remunerate}
\end{proposition}

\begin{proof}
The fact that the map $\phi_p$ is an isomorphism follows from a standard isomorphism theorem for vector
spaces. The fact that the remaining maps are well defined, and are  isomorphisms,
and also the commutativity of the diagrams are easily checked using 
Lemma~\ref{lem:functorial}, Proposition~\ref{prop:functorial}, 
and the definitions of the various subspaces involved.
\end{proof}

\begin{definition}[Support of a chain]
\label{def:support}
For $z = \sum_{\sigma \in K^{\sqbracket{p}}}c_\sigma \cdot \sigma \in C_p(K)$, we denote 
$\supp(z) = \{\sigma \in K^{\sqbracket{p}} \mid c_\sigma \neq 0\}$, and call $\supp(z)$ the
\emph{the support of $z$}.
\end{definition}

The following definition subsumes the one in \cite{essential}.

\begin{definition}
\label{def:rep}
Let $b = (s,t, 1) \in \mathbf{B}_p(\mathcal{F})$ be a simple bar
of $\mathcal{F}$ (see Definition~\ref{def:barcode2}).
We define
\begin{eqnarray*}
\Rep(b) &=&  Z_p(K_s) \setminus 
\widetilde{M}^{s,\infty}_p(\mathcal{F})
\mbox{ if $t = \infty$} \\
&=& \widetilde{M}^{s,t}_p(\mathcal{F}) \setminus \widetilde{N}^{s,t}_p(\mathcal{F}) \mbox{ else}.
\end{eqnarray*}
We call $\Rep(b)$ \emph{the set of cycles representing the bar $b$}.
More precisely, 
for $z \in Z_p(K_s)$, $z \in \Rep(b)$, if and only if $z$ represents a non-zero element
in $P^{s,t}_p(\mathcal{F})$.
\end{definition}

\begin{definition}[The set of essential simplices associated to a simple bar]
\label{def:essential}
Let $b = (s,t, 1) \in \mathbf{B}_p(\mathcal{F})$ be a simple bar
of $\mathcal{F}$.
We define
$
\Sigma(b) = \bigcap_{z \in \Rep(b)} \supp(z).
$
We will call $\Sigma(b)$ \emph{the set of essential simplices of $b$}.
\end{definition}

\begin{definition}[Relative essential content]
\label{def:content}
Let $b = (s,t, 1) \in \mathbf{B}_p(\mathcal{F})$ be a simple bar
of $\mathcal{F}$.
For $z = \sum_{\sigma \in K^{\sqbracket{p}})} c_\sigma \cdot \sigma \in \Rep(b)$, we denote
$
\econt(z) = \left(\frac{\sum_{\sigma \in \Sigma(b)}  c_\sigma^2}{\sum_{\sigma \in K^{\sqbracket{p}}} c_\sigma^2}\right)^{1/2}.
$
We will call $\econt(z)$ \emph{the relative essential content of $z$}.
\end{definition}

\begin{definition}[Harmonic representative]
\label{def:harmonic-rep}
Given a simple bar $b = (s,t,1)$ we will call any non-zero cycle in $\mathcal{P}^{s,t}_p(\mathcal{F})$
\emph{a harmonic representative of $b$}.
\end{definition}

\begin{remark}
\label{rem:essential-empty}
\begin{figure}
    \centering
    \includegraphics[scale=0.50]{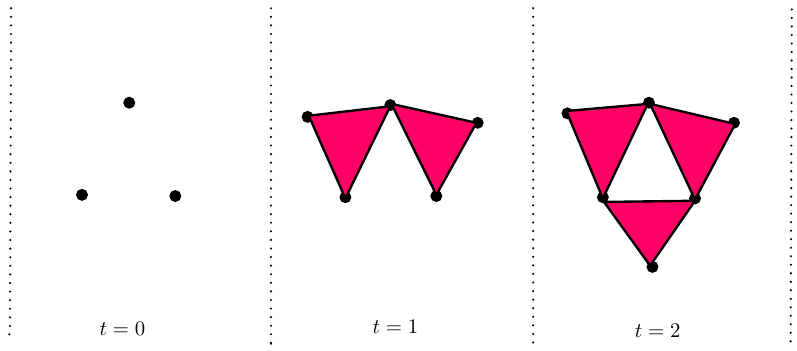}
    \caption{Empty set of essential edges}
    \label{fig:essential-empty}
\end{figure}
Note that the set of essential simplices of a bar can be empty. For example, the unique bar,
$(2,\infty,1) \in \mathbf{B}_1(\mathcal{F})$ of the filtration $\mathcal{F}$ shown in 
Figure~\ref{fig:essential-empty} has no essential simplices (edges). None of the edges 
are indispensable for obtaining a representative of the unique non-zero homology class in dimension one
that is born at $t=2$.
However, this situation cannot occur if the filtration is simplex-wise (see \cite{essential}). 
In that case the last simplex added
at the time a bar is created is always in the set of essential simplices of the bar. The filtration in Figure~\ref{fig:essential-empty} is not a simplex-wise filtration.
\end{remark}

\subsection{Relative essential content}
\label{subsec:harmonic-rep}

\begin{theorem}[Harmonic representatives maximize relative essential content]
\label{thm:essential}
Let $K$ be a finite simplicial complex and $\mathcal{F} = (K_t)_{t \in T}$ denote a filtration
of $K$. Suppose $p \geq 0$, and  
let $b = (s,t, 1) \in \mathbf{B}_p(\mathcal{F})$ be a simple bar.
Let $z_0$ be a harmonic representative of $b$. Then for any $z \in \Rep(b)$,
\[
\econt(z) \leq \econt(z_0).
\]
\end{theorem}

\begin{remark}
Note that Theorem~\ref{thm:essential} implies in particular that the relative essential contents
of any two harmonic representatives of a simple bar are equal. But this is clear also from the definition of the relative essential content and the fact that any two harmonic representatives of a simple bar are proportional.
\end{remark}

Before proving Theorem~\ref{thm:essential} we first prove a technical lemma.

\begin{lemma}
\label{lem:essential:2}
Let $V,W \subset Z_p(K)$ be subspaces  with $W \subset V$ 
with  $\dim V - \dim W =1$.
Let 
$
\Gamma =
\bigcap_{z \in V \setminus W} \supp(z)
$.
Then, for each  $z \in V \setminus W$
\begin{remunerate}
    \item 
    \label{itemlabel:lem:essential:2:1}
    $\Gamma \subset \supp(z)$;
    \item 
    \label{itemlabel:lem:essential:2:2}
    for each $\sigma \in \supp(z)$,  $\sigma \in \Gamma$ 
    if and only if 
    $\sigma \in W^\perp$.
\end{remunerate}
\end{lemma}

\begin{proof}
Part~\eqref{itemlabel:lem:essential:2:1} is clear. 

We now prove Part~\eqref{itemlabel:lem:essential:2:2}.
Let $z \in V \setminus W$,  and $\sigma_0 \in \supp(z)$.

\begin{remunerate}
\item
$\sigma_0 \in \Gamma \Rightarrow \sigma_0 \in W^\perp$:
Suppose  that $\sigma_0 \in \Gamma$ and  that 
$
\sigma_0 \not\in W ^\perp.
$
Then, there exists  
$
z'  \in W,
$
with $\la\sigma_0,z'\ra \neq 0$.
Otherwise, $\sigma_0$ is orthogonal to every cycle in 
$W$, and and hence $\sigma_0 \in W^\perp$.
Let
$
z'' = z - \frac{\la \sigma_0,z\ra}{\la\sigma_0,z'\ra} z'.
$
Then,
$z'' \in V \setminus W$ (otherwise $z \in W$)
but  $\sigma_0 \not\in \supp(z'')$.
This contradicts the fact that $\sigma_0 \in \Gamma$.
Hence,
$
\sigma_0 \in W^\perp
$.
\item
$\sigma_0 \in W^\perp \Rightarrow \sigma_0 \in \Gamma$:
Let
$
\sigma_0 \in W^\perp,
$
and let $z'' \in V\setminus  W$. Suppose that 
$\sigma_0 \not\in \supp(z'')$.

Since $\dim V - \dim W =1$, there exists $c \in \R, c\neq 0$, such that
$z + c z'' \in W$.

Then, 
$\sigma_0 \in \supp(z+c z'')$, but $z+ c z'' \in W$.
This contradicts the fact that 
$
\sigma_0 \in W^\perp
$.

So, $\sigma_0 \in \supp(z'')$ for every $z'' \in V \setminus W$. Hence, $\sigma_0 \in \Gamma$.
\end{remunerate}
\end{proof}

\begin{proof}[Proof of Theorem~\ref{thm:essential}]
There are two cases.

If $t = \infty$, set 
$
V = Z_p(K_s),
W = \bigcup_{s \leq t} \widetilde{M}^{s,t}_p(\mathcal{F}),
\Gamma = \Sigma(b).
$

If $t \neq  \infty$, set 
$
V = \widetilde{M}^{s,t}_p(\mathcal{F}), 
W = \widetilde{N}^{s,t}_p(\mathcal{F}), 
\Gamma = \Sigma(b).
$

Let 
$
z = \sum_{\sigma \in K^{\sqbracket{p}}} c_\sigma \cdot \sigma \in \Rep(b),
$,
$
z_1 = \sum_{\sigma \in \Sigma(b)} c_\sigma \cdot \sigma, 
$
and 
$
z_2 = z - z_1
$.
Using Lemma~\ref{lem:essential:2}, $z_1 \in W^\perp$. Clearly $z_1 \perp z_2$.
Also using Lemma~\ref{lem:essential:2} we have that
for each $\sigma \in \Sigma(b) = \supp(z_1)$,
$\sigma \in W^\perp$, 
from which it follows that 
$
\sigma \not\in \supp(\proj_W(z_2))
$.
Since it is clear from the definition of $z_2$, that
for each $\sigma \in \Sigma(b) = \supp(z_1)$, 
$
\sigma \not\in \supp(z_2)
$,
we obtain that
for each $\sigma \in \Sigma(b) = \supp(z_1)$,
\begin{equation}
\label{eqn:lem:essential:2:2.5}
\sigma \not\in \supp(z_2 - \proj_W(z_2)) = \supp(\proj_{W^\perp}(z_2)).
\end{equation}
In particular, this implies that
\begin{equation}
\label{eqn:lem:essential:2:3}
    z_1 \perp \proj_{W^\perp}(z_2).
\end{equation}

Let
$z_0 = \proj_{W^\perp}(z)$. Then, using Proposition~\ref{prop:harmonic-rep},
$z_0$ is a harmonic representative of $b$.

Moreover,
\begin{eqnarray*}
z_0 &=& \proj_{W^\perp}(z_1 + z_2) \\
&=& \proj_{W^\perp}(z_1) + \proj_{W^\perp}(z_2) \\
&=& z_1 + \proj_{W^\perp}(z_2),
\end{eqnarray*}
and \eqref{eqn:lem:essential:2:2.5} and \eqref{eqn:lem:essential:2:3} imply that
\begin{eqnarray}
\label{eqn:lem:essential:2:4}
||z_0||^2 &=& ||z_1||^2 + || \proj_{W^\perp}(z_2)||^2, \\
\label{eqn:lem:essential:2:5}
\econt(z_0)  &=&
\frac{||z_1||}{||z_0||}.
\end{eqnarray}

Hence,
\begin{eqnarray*}
\econt(z)^2 &=& \frac{||z_1||^2}{||z_1||^2 + ||z_2||^2} \\
&\leq& \frac{||z_1||^2}{||z_1||^2 + ||\proj_{W^\perp}(z_2)||^2} \mbox{ (projection does not increase length)} \\
&=& \frac{||z_1||^2} {||z_1+\proj_{W^\perp}(z_2)||^2}  \mbox{ (using \eqref{eqn:lem:essential:2:3})}\\
&=& \frac{||z_1||^2} {||z_0||^2} \mbox{ (using \eqref{eqn:lem:essential:2:4})}\\
&=& \econt(z_0)^2 \mbox{ (using  \eqref{eqn:lem:essential:2:5})}.
\end{eqnarray*}
This completes the proof of the theorem.
\end{proof}

\section*{Acknowledgement}
We thank the referees for carefully reading the manuscript and making valuable suggestions. We also thank Aziz Burak Guelen
for a careful reading and pointing out an error in the formulation of Proposition~\ref{prop:tilde-P} in a previous version of the paper.
The authors were partially supported by NSF grants CCF-1910441 and CCF-2128702.
\bibliographystyle{amsplain}
\bibliography{bibliography}

\end{document}

%% file: ex_shared.tex
\usepackage{lipsum}
\usepackage{amsfonts}
\usepackage{graphicx}
\usepackage{epstopdf}
\usepackage{algorithmic}
\usepackage{convex}
\newtheorem{notation}{Notation}[section]

\newtheorem{example}{Example}[section]

\ifpdf
  \DeclareGraphicsExtensions{.eps,.pdf,.png,.jpg}
\else
  \DeclareGraphicsExtensions{.eps}
\fi

\usepackage{enumitem}
\setlist[enumerate]{leftmargin=.5in}
\setlist[itemize]{leftmargin=.5in}

\newsiamremark{remark}{Remark}
\newsiamremark{hypothesis}{Hypothesis}
\crefname{hypothesis}{Hypothesis}{Hypotheses}
\newsiamthm{claim}{Claim}
\newsiamthm{convention}{Convention}

\headers{Harmonic persistent homology}{Saugata Basu and Nathanael Cox}

\title{Harmonic persistent homology}

\author{Saugata Basu\thanks{Department of Mathematics, Purdue University, 150 N. University St., West Lafayette IN, USA 
  (\email{sbasu@purdue.edu}, \url{https://www.math.purdue.edu/\~sbasu/}).}
\and Nathanael Cox\thanks{Department of Mathematics, Purdue University, 150 N. University St., West Lafayette IN, USA 
  (\email{cox175@purdue.edu}).}}

\usepackage{amsopn}
